\documentclass[11pt]{article}
\usepackage{amsmath}
\usepackage{amsfonts}
\usepackage{amssymb}
\usepackage{multirow}
\usepackage{mathrsfs}
\usepackage{cite}
\usepackage{ascii}
\usepackage{bbding}
\usepackage{pifont}
\usepackage{ifsym}
\usepackage{wasysym}
\usepackage{upgreek}
\usepackage{hyperref}
\usepackage{graphicx}
\usepackage{eso-pic}

\usepackage{nonfloat}
\usepackage[numbers]{natbib}

\renewcommand{\Box}{\framebox{\rule{0.3em}{0.0em}}}

\newtheorem{thm}{Theorem}[section]
\newtheorem{lema}{Lemma}[section]
\newtheorem{prop}{Proposition}[section]
\newtheorem{ex}{Example}[section]

\newtheorem{defi}{Definition}[section]

\newtheorem{cor}{Corollary}[section]

\newcommand{\bgeqn}{\begin{eqnarray}}
\newcommand{\edeqn}{\end{eqnarray}}

\newcommand{\bgeq}{\begin{eqnarray*}}
\newcommand{\edeq}{\end{eqnarray*}}
\newcommand{\bgc}{\begin{center}}
\newcommand{\edc}{\end{center}}

\newcommand{\beq}{\begin{equation}}
\newcommand{\eeq}{\end{equation}}
\newcommand{\beqa}{\begin{eqnarray}}
\newcommand{\eeqa}{\end{eqnarray}}
\newcommand{\beqas}{\begin{eqnarray*}}
\newcommand{\eeqas}{\end{eqnarray*}}
\newcommand{\ba}{\begin{array}}
\newcommand{\ea}{\end{array}}
\newcommand{\bi}{\begin{itemize}}
\newcommand{\ei}{\end{itemize}}

\def\dist{\mathop{\rm dist}}

\def\cL{{\cal L}}
\def\cC{{\cal C}}
\def\cN{{\cal N}}
\def\cI{{\cal I}}

\def\cF{{\cal F}}
\def\cB{{\cal B}}
\def\cM{{\cal M}}

\def\cS{{\cal S}}
\def\cX{{\cal X}}

\def\gph{{\rm gph}}
\def\bx{{\bar x}}
\def\by{{\bar y}}

\def\co{{\rm co}\,}

\def\Re{{\rm I\!R}}

\renewcommand{\Box}{\hfill \rule{2.3mm}{2.3mm}}

\makeatletter

\@addtoreset{equation}{section} \makeatother
\newenvironment{proof}{\noindent{\bf Proof. }}{\hfill $\Box$\medskip}

\title{Sensitivity analysis of the maximal value function  with applications in  nonconvex minimax programs}

\author{Lei Guo\thanks{\baselineskip 9pt School of Business, East China University of Science and Technology, Shanghai 200237, China. E-mail: lguo@ecust.edu.cn.}, \ \ Jane J. Ye\thanks{Department of Mathematics and
Statistics, University of Victoria, Victoria, BC, V8W 2Y2, Canada. E-mail: janeye@uvic.ca.}, \ \   Jin Zhang \thanks{Corresponding author. Department of Mathematics,
SUSTech International Center for Mathematics, Southern University of Science and Technology. National Center for Applied Mathematics Shenzhen, Peng Cheng
Laboratory, Shenzhen, Guangdong, China. E-mail: zhangj9@sustech.edu.cn.} }

\begin{document}

\maketitle

\baselineskip 16pt

\begin{abstract}
In this paper, we perform sensitivity analysis for the maximal value function which is the optimal value function for a parametric maximization problem. Our aim is to study various subdifferentials for the maximal value function. We obtain upper estimates of Fr\'{e}chet, limiting, and horizon subdifferentials of the maximal value function by using some sensitivity analysis techniques sophisticatedly. The derived upper estimates depend only on the union of all solutions and not on its convex hull or only one solution from the solution set. Finally, we apply the derived results to develop some new necessary optimality conditions for nonconvex minimax problems. In the nonconvex-concave setting, our Wolfe duality approach compare favourably with the first order approach in that the necessary condition is sharper and the constraint qualification is weaker.
\end{abstract}
\vspace{4pt}

\vspace{4pt}

\baselineskip 16pt

\section{Introduction}

The sensitivity analysis of the value function  for a parametric optimization problem is  very useful in many fields such as bilevel optimization, minimax programming, stochastic programming, robust optimization, dynamic programming, comparative statics analysis in economics.

In this paper we consider the following parametric maximization problem
\begin{eqnarray*}
(P_x)~~~~~~
\max\limits_y \ f(x,y) \ \
{\rm s.t. } \ \ y\in {\cal F}(x):=\{y\in \Re^{m}: g(x,y)\leq 0\},
\end{eqnarray*}
where $f:\Re^{n+m}\to \Re, g:\Re^{n+m}\to \Re^q$. Unless otherwise specified, we assume that $f(x,y),g(x,y)$ are continuously differentiable functions. Define the value function for the maximization problem or the maximal value function as
\[
V(x) := \sup\{f(x,y): y\in \cF(x)\},
\]
where by convention $V(x)=-\infty$ if $\cF(x)=\emptyset$, and the optimal solution map  as
\[
\cS(x):=\{y\in \cF(x): f(x,y)=V(x)\}.
\]
The  maximal value function  $V(x)$ is obviously related to the corresponding minimal value function  defined by
$$v(x):= \inf\limits_y \{ -f(x,y): y\in {\cal F}(x)\}$$ in the following way
\begin{eqnarray}
V(x)=-v(x).\label{negrelation}
\end{eqnarray}

The first goal of this paper is to study various subdifferentials of the maximal value function. The optimal value function of an optimization problem is intrinsically nonsmooth even if all involved functions are smooth. Generalized differential property of the value function is intensively studied in the literature. One stream focuses on the directional derivatives of the value function
(e.g., Bonnans and Shapiro \cite{BonShap}, Fiacco and Ishizuka \cite{Fiacco-Ishizuka}, Gauvin and Dubeau \cite{Gauvin-Dubeau}, Janin \cite{Janin}).
The other stream investigates the subdifferentials of the value function
(e.g., Clarke \cite{Clarke}, Gauvin \cite{Gauvin}, Guo et al.   \cite{sensitivity}, Lucet and Ye \cite{Lucet-Ye,Lucet-Ye2001}, Mordukhovich et al. \cite{Mordukhovich-Nam-Yen}, Thibault \cite{Thibault}, Ye and Zhang \cite{jinjane}).
In this paper, we focus on the latter stream partly because we will apply the derived results to develop necessary optimality conditions for minimax problems.

An upper estimate of the limiting subdifferential of the minimal value function  is relatively easier to obtain
(e.g., Clarke \cite{Clarke}, Guo et al. \cite{sensitivity}, Lucet and Ye \cite{Lucet-Ye,Lucet-Ye2001}).
By using the relationship (\ref{negrelation}), one hopes to obtain an upper estimate of the limiting subdifferential  of the maximal value function through the one for the corresponding minimal value function. This can be done if the maximal value function $V(x)$ is locally  Lipschitz continuous. Indeed,  it is well-known that if a function is locally Lipschitz continuous, then its Clarke subdifferential is homogenous, i.e.,
$$\partial^c  V(x)= \partial^c (-v)(x) = -\partial^c v(x),$$
while the limiting subdifferential is not homogeneous, i.e.,
$$\partial  V(x)=\partial (-v) (x)\not =-\partial v (x),$$
where $\partial$ and $\partial^c$ denote  the limiting and the Clarke subdifferential operators, respectively.
Hence if one needs to study the limiting subifferential of the maximal value function $V(x)$ through the one for the minimal value function $v(x)$, one could first ensure that the value function is locally Lipschitz continuous and then use the
following upper estimation
\begin{equation}\label{relation}
\partial V(x)\subseteq \partial^c  V(x)=\partial^c (-v)(x)=-\partial^c  v(x) =-\co \partial v(x),
\end{equation}
where $\co C$ denotes the convex hull of set $C$. In general, the inclusion in (\ref{relation}) may be strict unless the function $V(x)$ is Clarke regular (e.g., when the constraint set is a fixed compact subset as in Danskin's theorem, or when the value function is smooth).

In general, an upper estimate of the limiting subdifferential
$\partial v(x)$ involves the union of all solutions from the solution set ${\cal S}(x)$ (see e.g., (\ref{upper-estimate}) in Section \ref{sec:sub}) and hence any vector in $\partial v(x)$ may have an expression using {\it only one} solution from ${\cal S}(x)$. However, by \eqref{relation} the upper estimate for the limiting subdifferential
$\partial V(x)$ would involve the convex hull of the union of all solutions from ${\cal S}(x)$ and thus any vector in $\partial V(x)$ will depend on {\it all} solutions. Therefore using the limiting subdifferential would be complicated and even prohibitive in practice.
In this paper we aim at studying {the possibility of obtaining an estimate for the limiting subdifferential $\partial V(x)$ with the union of all solutions or only one solution from the solution set.}
We
first consider the case when $V(x)$ is locally Lipschitz continuous. In this case, the relationship \eqref{relation} holds. We will give some sufficient conditions for deriving an easily computable convex upper estimate for $\partial v(x)$. In this case, the convex hull operation would be superfluous and an easily computable upper estimate for $\partial V(x)$ would be derived.
Second, we consider the case when the value function is not necessarily locally Lipschitz continuous and we still wish to express any vector in the limiting subdifferential of the maximal function with only one solution from the solution set. Assuming the functions $-f(x,y),g(x,y)$ are convex in $y$, we employ the Wolfe dual (Wolfe \cite{wolfe}) to equate the maximal value function $V(x)$ with the minimal value function of the Wolfe dual for problem $(P_x)$ (see (\ref{ValuefunctionD})). Since the Wolfe dual of the maximization problem is a minimization problem, we do not need to employ the Clarke subdifferential to obtain an upper estimate for  $\partial V(x)$. {Consequently we have successfully obtained an upper estimate for the limiting or horizon subdifferential of the maximal value function $V(x)$ with the union of all solutions from the solution set and we can express any vector in the limiting or horizon subdifferential of the maximal value function with only one solution from the solution set.}

The minimax problem is an important research focus in mathematics, economics and computer science
(e.g., Von Neumann\cite{Neumann}, Von Neumann and Morgenstern \cite{Neumann-Morgenstern}, Sion \cite{Sion}). Most of the existing results for minimax problems focus on the convex-concave setting and study smooth minimax problem of the form
\begin{equation}
(P_{\rm minimax})~~~~~\min_{x\in X} \max_{y\in Y} f(x,y), \label{minimaxnew}
\end{equation}
where the function $f(\cdot, y)$ is convex for every $y\in Y$ and $f(x,\cdot)$ is concave for every $x\in X$, and $X,Y$ are closed convex sets. In this setting, the celebrated minimax theorem for the convex-concave setting (Sion \cite[Corollary 3.3]{Sion}) says that if either $X$ or $Y$ is compact, then
\begin{equation*}\label{minimax}
\min_{x\in X} \max_{y\in Y} f(x,y) = \max_{y\in Y} \min_{x\in X} f(x,y).
\end{equation*}
This property means that the game between two players (the player with decision $x$ and the one with decision $y$) is simultaneous and one can use the Nash equilibrium to define the optimality which is the well-known notion of optimality in simultaneous-move games.
We say that
 $(\bar x,\bar y)\in X \times Y$ is an optimal solution to the minimax problem $(P_{\rm minimax})$ in the sense of Nash (equilibrium) if and only if
$$
f(\bar x,y)  \le f(\bar x,\bar y) \le  f(x,\bar y),\quad \forall (x,y)\in X\times Y.
$$
In the convex-concave setting, the  necessary and sufficient first-order stationary condition in the sense of Nash equilibrium is as follows
\begin{equation}
0\in \nabla_x f(\bar x,\bar y)+\cN_X(\bar x),  \qquad  0\in -\nabla_y f(\bar x,\bar y)+\cN_Y(\bar y),\label{Nash}
\end{equation}
where $\cN_C(c)$ denotes the normal cone to set $C$ at point $c$ in the sense of convex analysis.

However many minimax problems in practice are not in the convex-concave setting. In particular many machine learning tasks {that appear in recent years} can only be formulated as $(P_{\rm minimax})$ with $f$ being nonconvex in $x$.
 Examples include generative adversarial networks (GAN) (Goodfellow et al. \cite{GAN2014}), distributionally robust learning, learning with non-decomposable loss, and Robust learning from multiple distributions (see  Rafique et al. \cite{RLLY} and the references within).
 We now consider the general minimax problem (\ref{minimaxnew}) where the function $f(x, y)$ is continuously differentiable and  $X,Y$ are closed  sets.
Under this general setting,
the order of the ``min" and the ``max" cannot be interchanged and hence the game is a sequential-move one. For a sequential-move game, it makes more sense to use the concept of a Stackelberg equilibrum. We say that $(\bar x, \bar y)$ with $\bar x\in X, \bar y\in Y$ is a locally optimal solution of $(P_{\rm minimax})$ in the sense of Stackelberg (equilibrum) if $\bar y$ is a globally optimal solution of the inner maximization problem
$\max_{y\in Y} f(\bar x,y)$ and $\bar x$ is a locally optimal solution of the problem
 $$(P_V)~~~~~~\min_{x\in X} V (x),$$
 where $V(x)=\sup_{y\in Y} f(x,y)$.
Note that in the  notion of optimality in the sense of Stackelberg, $\bar y$ must be a global solution of the inner maximization problem. This is the kind of optimality we will study in this paper.  Recently, concepts of local optimality for minimax problems where $\bar y$ is a local solution of the inner maximization problem have been studied; see e.g.,  \cite{Jin,DaiZhang,Jie}.

When the restricted sup-compactness condition (Definition \ref{Defn3.2}) holds at a solution $\bar x$ (e.g. when
$Y$ is a compact set),
the value function is  Lipschitz continuous  around  $\bar x$ and
$$\partial V(\bar x)\subseteq  \co \left \{ \nabla_x f(\bar x,  y'): y'\in \cS(\bar x)\right \}.$$
 Hence if $\bar x$ is a locally optimal solution of problem $(P_V)$, then  we have
\begin{equation}\label{Fermat}
0\in \partial V (\bx) +\cN_X(\bx),
\end{equation} where $\cN_C(c)$ denotes the limiting normal cone to set $C$ at point $c$, by which we obain
\begin{equation}\label{Fermatnew}0\in \co \left \{ \nabla_x f(\bar x,  y'): y'\in \cS(\bar x)\right \}+ \cN_X(\bx).\end{equation}
Moreover by Caratheodory's theorem,
there exist
$$
\lambda_i\geq 0, y_i\in \cS(\bar x)\ i=1,\ldots, n+1 \ {\rm with}\ \sum_{i=1}^{n+1}\lambda_i=1,
$$
such that
\begin{equation}
0 \in  \sum_{i=1}^{n+1} \lambda_i \nabla_x f(\bar x,  y_i)+ \cN_X(\bx), \quad 0\in -\nabla_y f(\bar x, y_i)+  \cN_Y(y_i),\ \ i=1,\ldots, n+1.
\label{caratheodory}
\end{equation}
However this kind of necessary optimality condition is not very practical. It may be used to verify that a feasible solution is not a locally optimal solution. But even this verification is not easy since if involves the whole solution set $\cS(\bar x)$.

Since minimax problems are special cases of bilevel programs or semi-infinite programs, the first order approach by which one replaces the inner maximization problem by its first order optimality condition, may be used. For the case where $f(x,y)$ is  concave in variable $y$ and $Y$ is a convex set,  by the first order approach, $(P_{\rm minimax})$ is equivalent to
$$\min_{x\in X,y\in Y} \ f(x,y)\quad
\mbox{ s.t. }  0\in  \nabla_y f(x,y)+\cN_Y(y) .$$
Let $(\bar x,\bar y)$ be a locally optimal solution of  $(P_{\rm minimax})$ in the sense of Stackelberg. If  $\bar y $ lies in the interior of set $Y$, then $(\bar x,\bar y)$ is a locally optimal solution of
$$\min_{x\in X, y\in \Re^m} \ f(x,y)\
\mbox{ s.t. }  \nabla_y f(x,y)=0.$$
 Further assume that $f$ is second-order continuously differentiable and a constraint qualification holds at  $(\bx,\by)$ for the equality constraint system
$
\nabla_y f(x,y)=0.
$
Then by the first order necessary optimality condition, there exists  a multiplier $\bar u$ such that
\begin{equation}\label{MPEC}
\begin{aligned}
0\in \nabla_x f(\bx,\by) + \nabla_{xy}^2 f(\bx,\by)\bar u +\cN_X(\bar x),\\
 \nabla_{yy}^2 f(\bx,\by)\bar u=0,\  \nabla_y f(\bar x,\bar y)=0.\end{aligned}
 \end{equation}

 Comparing two stationarity systems (\ref{caratheodory}) and (\ref{MPEC}) for the optimality in the sense of Stackelberg, one can see that in general all solutions from the solution set $\cS(\bar x)$ are needed in system (\ref{caratheodory}) while the system (\ref{MPEC}) only needs one solution but the concavity of the inner maximization problem is required and the second order derivatives are involved.
Unlike system (\ref{caratheodory}), (\ref{MPEC}) is a track-able  system of variational inequality in variables $x,y,u$.

In this paper we  study  the minimax problem  $\min\limits_{x\in X}\max\limits_{y\in \cF(x)} f(x,y)$. For this general minimax problem, we obtain the following two new results.
\begin{itemize}
\item[(i)]
Consider the nonconvex-concave setting where the inner maximization problem is concave. In this case, mathematical program with equilibrium constraints (MPEC) approach can be used. But if we use the Karush-Kuhn-Tucker (KKT) condition to replace the inner maximization problem and minimize over the original variables and the multipliers, the locally optimal solution of the
reformulated problem does not correspond to that of the original minimax problem (see Dempe and Dutta \cite{Dem-Dut} in the context of bilevel programs). Moreover the constraint qualification for the corresponding MPEC is in general very strong (see e.g., Gfrerer and Ye \cite{HY2018}).
In this paper by using the Wolfe duality approach,
we show that under reasonable constraint qualifications, a locally optimal solution of the minimax problem in the sense of Stackelberg is a strong (S-) stationarity condition, which is known to be the strongest stationary condition for MPECs.
Hence our new approach allow us to derive shaper optimality conditions under weaker constraint qualifications than the MPEC approach (see Theorems \ref{thm-mi}, \ref{thm-mima} and \ref{thm-com}).
\item[(ii)] For the general nonconvex-nonconcave setting, under either the inner semi-continuity of the solution mapping $\cS(x)$ or the concavity of the value funtion $V(x)$, we show that the  optimality condition depends on only one solution from the solution set $\cS(\bar x)$ (see Theorems \ref{thm5.3} and \ref{thm5.4}).  Applying to the case where $\cF(x)=Y$, we have that under either the inner semi-continuity of the solution mapping $\cS(x)$ or the joint concavity of $f(x,y)$ with convexity of $Y$, the optimality condition in terms of Stackelberg coincides with the one in terms of Nash. That is, system (\ref{Nash})
holds with $\cN_X(\bar x)$ being the limiting normal cone to set $X$ at $\bar x$.

\end{itemize}

The remainder of this paper is organized as follows. In Section \ref{sec:pre}, we give some preliminaries and preliminary results  needed in the paper.
In Section \ref{sec:sub}, we give upper estimates for the Fr\'{e}chet, limiting, and horizon subdifferentials of the maximal value function.
In Section \ref{sec:minimax}, we develop necessary optimality conditions for nonconvex minimax problems. We also apply our new optimality conditions to a special case of GAN to illustrate the advantage of our approach.    Section \ref{sec:con} concludes the paper.

\section{Preliminaries and preliminary results}\label{sec:pre}

\subsection{Basic notation}
For any $x\in \Re^n$, $\|x\|$ denotes its Euclidean norm in $\Re^n$. For $x\in \Re^n$ and $\delta>0$, we denote by $\cB_\delta(x):=\{y:\|y-x\|< \delta\}$
the open ball
 centered at $x$ with radius $\delta$ and {$\mathbb{B}$} the closed unit ball.
The boundary and interior of a subset $\Omega\subseteq \Re^n$ are denoted by ${\rm bdy} \Omega$ and ${\rm int} \Omega$, respectively. For a vector-valued function $\phi=(\phi_1,\ldots, \phi_m):\Re^n\to\Re^m$ and a point $x\in \Re^n$, we denote by $\cI_\phi(x):=\{i: \phi_i(x)=0\}$ the active index set at $x$ and  by $\nabla \phi(x)\in \mathbb{R}^{n\times m}$  the transposed Jacobian of $\phi$ at $x$.
For a set-valued map $\Phi:\Re^n \rightrightarrows\Re^m$, we denote its graph by $\gph\Phi :=\{(x,y): y\in \Phi(x)\}$. Given a set $\Omega\subseteq \Re^n$ and
a point $x\in \Re^n$, the distance from $x$ to $\Omega$ is denoted by
$\dist(x,\Omega):=\mbox{inf}\,\{\|y-x\|: y\in \Omega\}.$ For two vectors $a,b\in \Re^n$, the relation that $a$ is perpendicular to $b$ is denoted as $a\ \bot\ b$.
For a sequence $\{a^k\}_{k=1}^\infty$ and a point $a$, the condition that $\{a^k\}_{k=1}^\infty$ converges to $a$ as $k\to\infty$ is denoted by $a^k\to a$.
\textcolor{red}{
}

\subsection{Variational analysis}

We give some background materials on variational analysis in this subsection; see, e.g., Clarke \cite{Clarke}, 
Mordukhovich \cite{mor-book},
Rockafellar and Wets \cite{Rock98} for more details.

 For a function $\varphi:\Re^n\to[-\infty,\infty]$ and a point $x^*\in \Re^n$ where $\varphi(x^*)$ is finite, the
regular (or Fr\'{e}chet) subdifferential of $\varphi$ at $x^*$ is defined
by
$$\widehat{\partial}\varphi(x^*):=\left \{v: \varphi(x)\geq \varphi(x^*)+v^\top(x-x^*)+o(\|x-x^*\|)\right \},$$
the limiting (or Mordukhovich) subdifferential of $\varphi$ at $x^*$ is
defined by
$$\partial\varphi(x^*):= \left \{v: \exists x^k\to_\varphi x^*, \exists v^k\in \widehat{\partial}\varphi(x^k) \quad {\rm s.t.}\ v^k\to v \right \},$$
and the horizon (or singular Mordukhovich) subdifferential of $\varphi$ at
$x^*$ is defined by
$$\partial^{\infty}\varphi(x^*):=\left \{v: \exists x^k\to_\varphi x^*, v^k\in \widehat{\partial}\varphi(x^k)\ {\rm and}\ t_k\downarrow 0, \quad {\rm s.t.}\ t_kv^k\to v \right \},$$
where $o(\cdot)$ means $o(\alpha)/\alpha \to 0$ as $\alpha \downarrow 0$, and $x^k \to_\varphi x^*$ means that $x^k\to x^*$ and $\varphi(x^k)\to\varphi(x^*)$. It is well-known that a lower semi-continuous function $\varphi$ is Lipschitz continuous at $x$ if and only if $\partial^\infty \varphi(x)=\{0\}$ by Rockafellar and Wets \cite[Theorem 9.13]{Rock98}. If $\varphi$ is Lipschitz continuous at $x^*$, then  the Clarke subdifferential at $x^*$ is
$$\partial^c \varphi(x^*)=\co \partial \varphi(x^*).$$

The regular normal cone to a set $\Omega$ at $x^*\in \Omega$ is a closed cone defined by
$$\widehat{\cal N}_\Omega(x^*):=\{\xi\,:\  \xi^T (x-x^*) \leq o(\|x-x^*\|) \ {\rm for \ each} \ x\in \Omega\},$$
and the limiting normal cone to $\Omega$ at $x^*\in \Omega$ is a closed cone defined by
$${\cal N}_\Omega(x^*):=\{\xi \,: \ \xi=\lim_{k\to\infty}\xi^k\ {\rm with}\  \xi^k\in \widehat{\cal N}_\Omega(x^k)\ {\rm and} \ x^k\to_{\Omega} x^* \},$$
where $\ x^k\to_{\Omega} x^*$ means $x^k\rightarrow x^*$ and $x^k\in \Omega$ for all $k$.

\subsection{Constraint qualifications}


In this subsection, we introduce a unifying constraint qualification (CQ) that will be used in the subsequent analysis. We first review some CQs and their properties for the following  parametric nonlinear program
\begin{eqnarray}\label{nlp}
({\cal P}_p)~~~~~~~~~~\begin{array}{rl}
\max\limits_x & f(p,x)\\[4pt]
{\rm s.t.} & x\in \cX(p):=\{x: g(p,x)\le0, h(p,x)=0\},
\end{array}
\end{eqnarray}
where $f: \Re^m \times \Re^n \to  \Re,\ g:\Re^{m}\times\Re^n \to \Re^p$ and $h:\Re^{m}\times\Re^n \to \Re^q$ are continuously differentiable.

Constraint qualifications  are conditions  imposed on constraint systems such that local minimizers satisfy Karush-Kuhn-Tucker (KKT) conditions. When the constraint functions are all affine,  no extra condition is required for KKT conditions to hold at an optimal solution while for nonlinear constraint systems, certain CQ is required. The weakest CQ is the so-called Guignard's CQ but it is hard to verify. The verifiable CQs include the well-known linearly independent CQ (LICQ), Mangasarian-Fromovitz (MFCQ) and Slater's condition.
Some weaker and verifiable CQs are also proposed in the literature such as the  constant rank constraint qualification (CRCQ), the relaxed constant rank constraint qualification (RCRCQ), the relaxed constant positive linear dependence (RCPLD), and the quasi-normality
(Andreani et al. \cite{RCPLD}, Bertsekas and Ozdaglar \cite{BertsekasJOTA2002}, Minchenko and Stakhovski \cite{Minchenko-stakhovski}).
All the CQs mentioned above are stronger than the error bound property/metric subregularity/calmness; see, e.g., Ye and Zhou \cite[Theorem 7.4]{Ye-Zhou}.


Since the parametric nonlinear program  (\ref{nlp}) involves a parameter and so the constraint qualifications is dependent on the parameter. We now review  the parametric version of the CQs mentioned above.
\begin{defi}[LICQ, MFCQ, CRCQ, RCRCQ, RCPLD,  the quasi-normality]\label{defi}
\qquad

\noindent {\rm (i)} We say that LICQ holds at $\bar x \in  \cX(\bar p)$  if the family of partial gradient vectors
\begin{equation}\{\nabla_x g_i(\bar p,\bx): i\in\cI_g(\bar p,\bx)\} \cup \{\nabla_x h_j(\bar p,\bx): j=1,\dots, q\},\label{family}
\end{equation} where $\cI_g(\bar p,\bx):=\{i: g_i(\bar p,\bx)=0\}$, is linearly independent.

\noindent {\rm (ii)} We say that   MFCQ holds at $\bar x \in  \cX(\bar p)$  if the family of vectors (\ref{family})  is positively linearly independent. That is,
\begin{eqnarray*}
&&  \nabla_x g(\bar p,\bar x)\lambda^g+\nabla_x h(\bar p,\bar x)\lambda^h=0, \ 0\leq \lambda^g \perp - g(\bar p,\bar x)\geq 0  \Longrightarrow (\lambda^g,\lambda^h)=0.
\end{eqnarray*}

\noindent {\rm (iii)} We say that  CRCQ holds at $\bar x \in  \cX(\bar p)$ if there is $\delta>0$ such that for all index sets $\cI_1\subseteq \cI_g(\bar p,\bar x)$ and $\cI_2\subseteq \{1,\ldots,q\}$, the family of partial gradient vectors $$\{\nabla_x g_i(p,x), \nabla_x h_j(p,x):i\in \cI_1, j\in \cI_2\}$$ has the same rank for all $(p, x)\in \cB_\delta(\bar p,\bar x)$.

\noindent {\rm (iv)} We say that RCRCQ holds at $\bar x \in  \cX(\bar p)$ if there is $\delta>0$ such that for all $\cI\subseteq \cI_g(\bar p,\bx)$, the family of partial gradient vectors $\{\nabla_x g_i(p,x), \nabla_x h_j(p,x):i\in \cI, j=1,\ldots,q\}$ has the same rank for all $(p,x)\in \cB_\delta(\bar p,\bx)$.

\noindent {\rm (v)} Let ${\cal J}\subseteq\{1,\ldots,q\}$ be such that $\{\nabla_x h_j(\bar p,\bar x):j\in{\cal J}\}$ is a basis for ${\rm
span}\,\{\nabla_x h_j(\bar p,\bar x): j=1,\ldots,q\}$. We say that  RCPLD holds at $\bar x \in  \cX(\bar p)$ if there exists
$\delta>0$ such that
\begin{itemize}
\item $\{\nabla_x h_j( p,x): j=1,\ldots,q\}$ has the same rank for each $(p,x)\in {\cal
B}_{\delta}(\bar p,\bx)$;
\item for each ${\cal I}\subseteq {\cI}_g(\bar p,\bx)$, if $\{\nabla_x g_i(\bar p,\bx): i\in{\cal
I}\}\cup \{\nabla_x h_j(\bar p,\bx): j\in{\cal J}\}$ is positively linearly dependent, then $\{\nabla_x
g_i(p,x), \nabla_x h_j(p,x): i\in {\cal I}, j\in {\cal J}\}$ is linearly dependent
for all $(p,x)\in{\cal B}_{\delta}(\bar  p,\bx)$.

\end{itemize}

\noindent {\rm (vi)} We say that the  quasi-normality holds at $\bar x \in  \cX(\bar p)$ if  there are no nonzero vectors $(\lambda,\mu)$ such that
\begin{itemize}
\item $ \nabla_x g(\bar p,\bx)\lambda  +  \nabla_x h(\bar p,\bx)\mu =0,\
0\le \lambda \ \bot \ -g(\bar p,\bx) \ge0$;
\item there exists a sequence  $\{(p^k, x^k)\}$ converging to $(\bar p,\bx)$  such that for all $k$,
\[g_i(p^k,x^k)>0\ {\rm if}\ \lambda_i>0,\ \mu_ih_i(p^k,x^k)>0\ {\rm if}\ \mu_i\neq0.\]
\end{itemize}

\end{defi}

\begin{defi}[Robinson stability](Gfrerer and Mordukhovich \cite[Definition 1.1]{HM})\label{rss}
We say that Robinson stability ${\rm (RS)}$  holds for the parametric feasible map  $\cX(p)$  at $(\bar p,\bar x)\in {\rm gph } \cX$  if there exist positive scalars $\kappa, \delta$ such that
\begin{equation*}
{\rm dist}(x, {\cX}(p))\leq \kappa \left[\sum_{i=1}^p \max\{g_i(p,x),0\} + \sum_{i=1}^q |h_i(p,x)|\right],  \quad \forall
(p,x) \in  \cB_\delta(\bar p,\bar x).\end{equation*}
\end{defi}
If all dependencies on the parameter $p$ are omitted (i.e.,  $g(p,x)$ and $h(p,x)$ are both independent of $p$) in the above inequality, then the Robinson stability becomes the local error bound property. Hence RS means that the error bound property holds with the same constant $\kappa$ for all parameters sufficiently close to $\bar{p}$.
Although the term RS was first given in Gfrerer and Mordukhovich \cite{HM}, this property was first studied by Robinson \cite{robinson76}. Other terminologies have been given to this property in the literature,  for example, in Minchenko and Stakhovski \cite{Minchenko-stakhovski} and some other publications, this property is referred to as R-regularity. For the case where  the solution map ${\cal S}(x)$ satisfies  RS, the property was called the uniform weak sharp minimum and the uniform parametric error bound in Ye and Zhu \cite{Ye-Zhu}
 and Ye et al. \cite{Ye-zhu-zhu} respectively.

 For the parametric nonlinear program  (\ref{nlp}),   CRCQ,  RCRCQ,   RCPLD and the quasi-normality are all weaker than  MFCQ which is weaker than  LICQ.  By adapting the proof of
Andreani et al. \cite[Theorem 1]{RCPLD} in the non-parametric case to the parametric case, it is easy to verify that   RCRCQ  implies   RCPLD for the parametric program.  Moreover RCPLD and the inner semi-continuity of $\cX(p)$ can imply  RS  (Mehlitz and Monchenko\cite[Theorem 3.3]{R-regularity}).

Moreover, some CQs such as LICQ are locally persistent in the sense that if they hold at a given point, then they will hold at all feasible point near such a point. But some CQs are not locally persistent such as Guignard's CQ.  For the parametric nonlinear program (\ref{nlp}), {we say that a CQ is locally persistent at $\bar{x}\in \cX(\bar{p})$ if it holds at any $x\in \cX(p)$ with $(x,p)$ in some neighborhood of $(\bx,\bar p)$.}


In the following, we introduce a unifying class of  CQs for the parametric nonlinear program by extending the multiplier stability in Guo et al.  \cite[Definition 3.3]{guo-stability} and combining it with the local persistence. For convenience of representation, for given points $\beta,x\in \Re^n$ and $p\in \Re^m$, we define the multiplier set
\begin{equation}\label{stand multi}
\cM(\beta,p, x):=\left\{(\lambda,\mu):\begin{array}{l}
\beta - \nabla_x g(p,x)\lambda - \nabla_x h(p,x) \mu =0 \\ [5pt]
0\le \lambda\ \bot -g(p,x)\ge0, \ h(p,x)=0
\end{array}\right\}.
\end{equation}
Note that when $\beta=\nabla_x f(p,x)$,  the multiplier set is the set of all KKT multipliers for the parametric optimization problem (\ref{nlp}) at $x\in \cX(p)$.

\begin{defi}[Stable parametric CQ]\label{defi1}
We say that a condition imposed on the constraint region of the parametric optimization problem (\ref{nlp}) is a stable parametric CQ at $\bar{x}\in \cX(\bar{p})$ if it is locally persistent at  $\bar{x}\in \cX(\bar{p})$; and
\begin{itemize}
\item[\rm (i)] if $\bar{x}\in \cX(\bar{p})$ is a local minimizer of problem (${\cal P}_{\bar{p}}$), then ${\cM}(\nabla_x f(\bar{p},\bar{x}), \bar{p},\bar{x})\neq \emptyset$;

\item[\rm (ii)]
for any given vector sequences $\{\beta^k\}_{k=1}^\infty\subseteq \Re^n$, $\{x^k\}_{k=1}^\infty\subseteq \Re^n$, and $\{p^k\}_{k=1}^\infty\subseteq \Re^m$ satisfying $\beta^k\to \bar{\beta}$, $x^k\to \bar{x}$ and $p^k\to \bar{p}$ as $k\to\infty$ and ${\cM}(\beta^k,p^k,x^k)\neq \emptyset$ for all $k$, there exists a multiplier sequence $\{(\lambda^k,\mu^k)\in {\cM}(\beta^k,p^k,x^k)\}_{k=1}^\infty$  converging to some $(\bar \lambda,\bar \mu)$ on a subsequence such that $(\bar \lambda,\bar \mu)\in {\cM}(\bar{\beta},\bar{p},\bar{x})$.
\end{itemize}
\end{defi}

It is easy to see that LICQ and MFCQ as defined in  Definitions \ref{defi} are both stable parametric CQs. The following result shows that all  conditions in Definitions \ref{defi} and \ref{rss} are stable parametric CQs in the sense of Definition \ref{defi1}.

\begin{prop}
Let $\bar{x}\in \cX(\bar{p})$.  Then   CRCQ, RCRCQ, RCPLD,  the quasi-normality and RS are all stable parametric CQs at $\bar{x}\in \cX(\bar{p})$.
\end{prop}
\begin{proof}
The local persistence of CRCQ, RCRCQ and RS follows from their definitions immediately. The local persistence of RCPLD and the quasi-normality respectively follows from adapting the proof of Andreani et  al  \cite[Theorem 4]{RCPLD} and Ozdaglar and Bertsekas \cite[Lemma 2]{OzBer} in the non-parametric case to the parametric case.

It follows from  Andreani et al. \cite[Corollay 1]{RCPLD} that  RCPLD is a CQ. Then the stronger  CRCQ and RCRCQ are also CQs.
By Ozdaglar and Bertsekas \cite[Proposition 1]{OzBer}, the  quasi-normality is a CQ.  The RS clearly implies the local error bound for the system $\{x: g(\bar{p},x)\le0, h(\bar{p},x)=0\}$. Thus, the Clarke calmness holds for problem \eqref{nlp} with $p=\bar{p}$.  Hence  RS is a CQ (Clarke \cite[Proposition 6.4.4]{Clarke}).

Using the proof technique of Theorem 2 of Andreani et al. \cite{RCPLD}, we can show that   Definition 2.3(ii) holds  if  RCPLD holds. Thus, the stronger CRCQ and RCRCQ also have this property.  We next show that  quasi-normality implies   Definition 2.3(ii).
For given points $\beta,x\in \Re^n$ and $p\in \Re^m$, let
\begin{equation}\label{stand multi-new}
\cM^e(\beta,p, x):=\left\{(\lambda,\mu):\begin{array}{l}
\beta - \nabla_x g(p,x)\lambda - \nabla_x h(p,x) \mu =0\\ [4pt]
 0\le \lambda \ \bot \ -g(p,x)\ge 0,\ h(p,x)=0 \\[4pt]
 {\rm there\ exists\ a\ sequence}\  \{(p^l, x^l)\}\ {\rm converging\ to}\ (p,x) \\ [4pt]
  {\rm such\ that\ for\ all}\ l,\\[4pt]
g_i(p^l,x^l)>0\ {\rm if}\ \lambda_i>0,\ \mu_ih_i(p^l,x^l)>0\ {\rm if}\ \mu_i\neq0
\end{array}\right\}
\end{equation}
be the enhanced multiplier set. When $\beta=\nabla_x f(p,x)$,  the enhanced  multiplier set is the set of all enhanced KKT multipliers for the parametric optimization problem (\ref{nlp}) at $x\in \cX(p)$. Let $\beta^k\to \bar{\beta}$, $x^k\to \bar{x}$ and $p^k\to \bar{p}$ as $k\to\infty$, and ${\cM}(\beta^k,p^k,x^k)\neq \emptyset$ for all $k$. By Giorgi et al.  \cite[Theorem 4.6]{Giorgi},  any KKT point is an enhanced KKT point which means that for each $(\lambda^k, \mu^k)\in {\cM}(\beta^k,p^k,x^k)$,
there exists $(\tilde \lambda^k, \tilde \mu^k)\in {\cM}^e(\beta^k,p^k,x^k)\subseteq {\cM}(\beta^k,p^k,x^k)$.
Then using the proof technique of Ye and Zhang \cite[Theorem 3]{jinjane}, we can show that  the sequence $\{ (\tilde \lambda^k, \tilde \mu^k)\}$  is bounded if the quasi-normality holds at $\bar x \in \cX(\bar{p})$.  Taking a subsequence we obtain $(\bar \lambda,\bar \mu)$  such that $(\bar \lambda,\bar \mu)\in {\cM}(\bar{\beta},\bar{p},\bar{x})$.
Finally we show that RS implies condition (ii) in Definition 2.3.  It follows from Gfrerer and Mordukhovich \cite[Theorem 3]{HM} that if RS holds at $(\bar{p},\bar{x})$, then a so-called bounded multiplier property holds at $(\bar{p},\bar{x})$, which means that there exists $\kappa>0$ such that for all sufficiently large $k$,
\[
{\cM}(\beta^k,p^k,x^k) \cap \kappa \|\beta^k\|\mathbb{B} \not= \emptyset.
\]
Let $(\lambda^k,\mu^k)\in {\cM}(\beta^k,p^k,x^k) \cap \kappa \|\beta^k\|\mathbb{B} $. Since $\beta^k\to \bar{\beta}$ as $k\to\infty$, it follows that $\{\beta^k\}$ is bounded. Thus, we have that $\{(\lambda^k,\mu^k)\}$ is bounded and it has a convergent subsequence. Since $(\lambda^k,\mu^k)\in {\cM}(\beta^k,p^k,x^k)$, it follows that
\begin{eqnarray*}
&& \beta^k - \nabla_x g(p^k,x^k)\lambda^k- \nabla_x h(p^k,x^k) \mu^k =0, \\ [4pt]
&& h(p^k,x^k)=0,\ 0\le \lambda^k\ \bot -g(p^k,x^k)\ge0.
\end{eqnarray*}
Taking limits on the subsequence on which $(\lambda^k,\mu^k)$ is convergent to $(\bar\lambda,\bar \mu)$, it follows that $(\bar \lambda,\bar \mu)\in {\cM}(\bar{\beta},\bar{p},\bar{x})$. The proof is complete.
\end{proof}

In the paper, we  also need to use the following concepts of CQs for nonparametric nonlinear programs.
\begin{defi}[Stable CQ] \label{defi2}
Let $\bar  x\in \cX:=\{x: g(x)\leq 0, h(x)=0\}$. We say that a stable CQ holds at $\bar x$ if all conditions in Definition \ref{defi1} holds with all dependencies on the parameter $p$ omitted.
\end{defi}
It should be noted that most  CQs for nonparametric nonlinear programs in the literature, namely, LICQ, MFCQ, CRCQ, RCRCQ, RCPLD, and the quasi-normality, are all stable CQs in the sense of Definition \ref{defi2}.

\section{The subdifferential of the maximal value function}\label{sec:sub}

In this section, we investigate upper estimates for the Fr\'{e}chet, limiting, and horizon subdifferentials of the maximal value function $V(x)$ with a single optimal solution or the union of all optimal solutions. We divide our analysis into two cases: the one when $V(x)$ is locally Lipschitz continuous and the one when $V(x)$ may not be locally Lipschitz continuous. For the first case, we will utilize the inner semi-continuity of the optimal solution map which is a relatively strong condition; whilst for the second case, the Wolfe dual will be employed.

For a given scalar $r\in \{0,1\}$, the generalized Lagrangian function for the maximization problem $(P_\bx)$ is defined as
$\cL^r(\bx,y,\lambda):=r f(\bx,y) - g(\bx,y)^\top\lambda.$
For the sake of simplicity, when $r=1$ we omit the superscript $r$. We denote the set of generalized multipliers for problem $(P_\bx)$ at $y\in \cS(\bx)$ as
\begin{equation*}\label{lagset}
\Sigma^r(\bar x, y):=\left \{ \lambda : \nabla_y \cL^r(\bx,y,{\lambda})=0,\ 0\le - g(\bx,y)\ \bot\ {\lambda} \ge0
\right \},
\end{equation*} and we omit the superscipt $r$ when $r=1$.

\subsection{The case when $V$ is locally Lipschitz continuous}\label{subsec}

As discussed in the introduction, when the maximal value function  is locally Lipschitz continuous, by \eqref{relation} the subdifferential of the maximal value function can be estimated as
\begin{equation}\label{estim}
\partial V(\bx) \subseteq \partial^c  V(\bx) \subseteq -\co \partial v(\bx).
\end{equation}
In general, the upper estimate for $\partial v(\bar x)$ takes the following form (e.g., Lucet and Ye \cite[Theorem 4.4]{Lucet-Ye2001}):
\begin{eqnarray}
\partial v(\bar x) &\subseteq &\bigcup_{y\in {\cal S}(\bar x)} \{-\nabla_x \cL(\bx,y,{\lambda}): {\lambda} \in \Sigma(\bar x, y)\}.\label{upper-estimate}
\end{eqnarray}
Hence by (\ref{estim})-(\ref{upper-estimate}), provided that $V$ is Lipschitz continuous at $\bx$, we have
\begin{equation}\label{estimnew}
\partial V(\bx) \subseteq \co \bigcup_{y\in {\cal S}(\bar x)} \{\nabla_x \cL(\bx,y,{\lambda}): {\lambda} \in \Sigma(\bar x, y)\}.
\end{equation}
In this subsection, we investigate sufficient conditions for ensuring that the upper estimate (\ref{upper-estimate}) can take the following simpler form
\begin{equation}\label{inclusion}
\partial v(\bar x) \subseteq \{-\nabla_x \cL(\bx,\by,{\lambda}): {\lambda} \in \Sigma(\bar x,\bar y)\},
\end{equation}
for a given $\bar y\in {\cal S}(\bar x)$.
Then due to the linearity of the function $\nabla_x\cL(x,y,{\lambda})$ with respect to $\lambda$ and the convexity of the multiplier set $\Sigma(\bar x,\bar y)$, the upper estimate in (\ref{inclusion})  is convex and hence the convex hull operation in (\ref{estimnew}) is superfluous.

The first result of this subsection depends on the inner semi-continuity which is a set-valued generalization for the continuity of a single-valued map.

\begin{defi}[Inner semi-continuity](Mordukhovich \cite[Definition 1.63]{mor-book}) Given $\bar y\in {\cal S}(\bar x)$, we say that the solution map ${\cal S}(x)$ is inner semi-continuous at $(\bar x,\bar y)$, provided that  for any sequence $ x^k \rightarrow \bar x$, there exists a sequence $y^k\in {\cal S}(x^k)$ converging to $\bar y$.
\end{defi}
From this definition, it is easy to see that if the solution map ${\cal S}(x)$ is lower semi-continuous at $\bar x$, then it is inner semi-continuous at $(\bar x,\bar y)$ for each $\bar y\in {\cal S}(\bar x)$. Moreover, by Mehlitz and Minchenko \cite[Lemma 2.2]{R-regularity}, the inner semi-continuity of ${\cal S}(x)$ at $(\bar x,\bar y)$ is guaranteed by the existence of  a uniform weak sharp minimum around  $(\bar x,\bar y)$ (see \cite{Ye-Zhu,Ye1998} for the definition and sufficient conditions).

We now show that under the inner semi-continuity, the upper estimate of $\partial V(\bar x)$ can be obtained by using any given solution $\bar y\in {\cal S}(\bx).$
\begin{thm}\label{thm4.1}
Let $\bar y\in {\cal S}(\bar x)$. Suppose that the solution map ${\cal S}(x)$ is inner semi-continuous at $(\bar x,\bar y)$. Assume that the local error bound holds for the system
$g(x,y)\leq 0$ at $(\bar x, \bar y)$.
Then $V(x)=-v(x)$ is continuous at $\bar x$ and
\begin{eqnarray}
\partial v(\bar x) &\subseteq& \{-\nabla_x \cL(\bx,\by,{\lambda}): {\lambda} \in \Sigma(\bar x,\bar y)\}, \label{eqn1}\\
\partial^\infty  v(\bar x) &\subseteq& \{\nabla_x g(\bx,\by) {\lambda}: {\lambda} \in \Sigma^0(\bar x,\bar y)\}.\label{eqn2}
\end{eqnarray}
Assume further that
\begin{equation}\label{sigular1}
\{\nabla_x g(\bx,\by) {\lambda}:  {\lambda} \in \Sigma^0(\bar x,\bar y)\}=\{0\}.
\end{equation}
Then $V$ is Lipschitz continuous at $\bx$ and ${\partial  V(\bx) \subseteq \partial^c  V(\bx)} \subseteq \left \{\nabla_x \cL(\bx,\by,{\lambda}): {\lambda} \in \Sigma(\bar x,\bar y) \right \}$.
\end{thm}
\begin{proof} 
Under the inner semi-continuity at $(\bx, \by)$, it is straightforward to verify that the  minimal value function $v(x)$ is continuous at $\bx$. Let $\theta(x,y) := -f(x,y) + I_{\gph\cF}(x,y)$ where $I_{\gph\cF}$ is the indicator function. Then $v
(x)=\inf_y \theta(x,y)$.
Since $\cS(x)$ is inner semicontinuous at $(\bx,\by)$, by Mordukhovich \cite[Theorem 1.108]{mor-book} and the sum rule of limiting subdifferential, we have
\begin{eqnarray*}
&& \partial v(\bar x)\subseteq \left \{x^* : (x^*,0)\in -\nabla f(\bx,\by) + \cN_{\gph\cF}(\bx,\by) \right \},\\
&& \partial^\infty v(\bar x) \subseteq \{x^* : (x^*,0)\in \cN_{\gph\cF}(\bx,\by)\}.
\end{eqnarray*}
Therefore, under the local error bound condition for the system $g(x,y)\leq 0$ at $(\bx,\by)$, by the explicit expression for $\cN_{\gph\cF}(\bx,\by)$ (e.g., Henrion et al. \cite[Theorem 4.1]{HJO} or Gfrerer and Ye \cite[Proposition 4]{HY2018})  the desired inclusions (\ref{eqn1})-(\ref{eqn2}) follow immediately.

Since (\ref{eqn2}) and (\ref{sigular1}) imply $\partial^\infty  v(\bar x)=\{0\}$,  it follows that $v(x)$ (and $V(x)$ as well)  is Lipschitz continuous
at $\bar x$ under condition (\ref{sigular1}) (Rockafellar and Wets \cite[Theorem 9.13]{Rock98}).
Thus  $$\partial V(\bar x) \subseteq \partial^c V(\bar x) \subseteq -\co \partial v(\bar x).$$
This together with condition \eqref{eqn1} implies the desired result.
\end{proof}

Recall that
 MFCQ holds at $(\bx, \by)$ for the system $g(x,y)\leq 0$ at $(\bx, \by)$ if and only if
\begin{equation}\nabla g(\bx,\by){\lambda}=0 , \quad 0\le -g(\bx,\by)\ \bot\ {\lambda}\ge0 \Longrightarrow {\lambda}=0,\label{full} \end{equation}
while  MFCQ holds at $\by$  for the system $g(\bar x,y)\leq 0$ if and only if
\begin{equation} \nabla_y g(\bx,\by){\lambda}=0 , \quad 0\le -g(\bx,\by)\ \bot\ {\lambda}\ge0 \Longrightarrow {\lambda}=0.\label{partial} \end{equation}
The implication relation (\ref{partial}) is clearly stronger than the relation (\ref{full}). Hence under condition  (\ref{partial})   the local error bound holds for $g(x,y)\leq 0$ at $(\bar x, \bar y)$ and moreover, (\ref{sigular1}) holds.
The following corollary follows immediately.

\begin{cor}\label{cor4.1}
Let $\bar y\in {\cal S}(\bar x)$.
Suppose that ${\cal S}(x)$ is inner semi-continuous at $(\bar x,\bar y)$ and   MFCQ holds  for the system $g(\bx,y)\leq 0$ at $ \by$. Then $V(x)$ is Lipschitz continuous at $\bx$ and $\partial V(x) \subseteq \{\nabla_x \cL(\bx,\by,{\lambda}): {\lambda} \in \Sigma(\bar x,\bar y)\}$.
\end{cor}

By Robinson \cite[Theorem 3.2]{Robinson82}, when MFCQ holds and second order sufficient condition holds, the optimal solution mapping is inner semi-continuous. Thus, the following result follows immediately.

\begin{cor}
Let $\bar y\in {\cal S}(\bar x)$. Assume that $-f(x,\cdot),g(x,\cdot)$ are convex functions for any $x$ near $\bx$. Suppose that MFCQ holds  for the system $g(\bx,y)\leq 0$ at $ \by$ and for any $\lambda\in \Sigma(\bar x, \bar y)$, the second-order sufficient condition holds at $(\bx, \lambda)$:
\[
d^\top \nabla_{xx}^2 \cL(\bx,\by,\lambda)d<0,\quad \forall d\in {\cal C}(\bx,\by)\backslash \{0\},
\]
where
\[
{\cal C}(\bx,\by)=\left\{d:\begin{array}{l} \nabla_y g_i(\bx,\by)^\top d=0 \ i\in \cI_g(\bx,\by)\ {\rm and}\ \lambda_i>0 \\[4pt]
\nabla_y g_i(\bx,\by)^\top d\le0 \ i\in \cI_g(\bx,\by)\ {\rm and}\ \lambda_i=0 \end{array}\right\}.
\]
Then $V(x)$ is Lipschitz continuous at $\bx$ and $\partial V(x) \subseteq \{\nabla_x \cL(\bx,\by,{\lambda}): {\lambda} \in \Sigma(\bar x,\bar y)\}$.
\end{cor}
\begin{proof}
By the convexity assumptions of the involved functions, it follows that globally optimal solutions are the same as the locally optimal solutions. Then by Robinson \cite[Theorem 3.2]{Robinson82}, the solution mapping $\cS$ is continuous and hence inner semi-continuous at $(\bar x,\bar y)$. Thus the desired result follows from Corollary \ref{cor4.1},
\end{proof}

We now identify another situation where the upper estimate of $\partial V(\bx)$ involves only an arbitrarily specified solution $\bar y$.

\begin{thm}\label{thm4.2}
Assume that $\bar y\in {\cal S}(\bar x)$, and $f(x,y)$ is jointly concave and $g(x,y)$ is jointly quasiconvex in $(x,y)$. Suppose that there is an open set ${\cal O} \ni \bar x$ such that $\cF(x)$ is nonempty and the objective function $f(x,y)$ is bounded above on $\cF(x)$. Assume that a CQ holds for the system $g(x,y)\leq 0
$ at $(\bx,\by)$.  Then $V(x)$ is concave, Lipschitz continuous at $\bx$,  and
\begin{eqnarray*}
\partial  V(\bar x) \subseteq  \left  \{\nabla_x \cL(\bx,\by,{\lambda}): {\lambda} \in \Sigma(\bar x,\bar y)\right \}.
\end{eqnarray*}
\end{thm}
\begin{proof}
Since $-f(x,y)$ is jointly convex and $g(x,y)$ is jointly quasiconvex in $(x,y)$, it follows that $v(x)$ is a convex function
(Fiacco and Kyparisis \cite[Proposition 2.1]{Fiacco-Kyparisis}). Then by Ye and Wu \cite[Proposition 4.1]{YeWu} we have
\begin{eqnarray*}
\partial  v(\bar x) = \left \{-\nabla_x \cL(\bx,\by,{\lambda}): {\lambda} \in \Sigma(\bar x,\bar y) \right \}.
\end{eqnarray*}
Under the assumptions, it is not hard to verify that $V(x)$ is finite around $\bar x$. This together with the convexity of $v(x)$ implies that $v(x)$ must be Lipschitz continuous at $\bx$ (Rockafellar \cite[Theorem 10.4]{rock-convex}). Then by \eqref{estim} and the convexity of Lagrange multiplier set, we have
$$\partial  V (\bar x)\subseteq - \co\partial  v(\bar x)=\left  \{\nabla_x \cL(\bx,\by,{\lambda}): {\lambda} \in \Sigma(\bar x,\bar y)\right \}.
$$
\end{proof}

\subsection{The case when $V$ may not be locally Lipschitz continuous}\label{subsection-sub2}

To obtain an easily computable convex upper estimate for the subdifferential of the minimal value function, in the previous subsection we assume some relatively strong assumptions that the optimal solution map is inner semi-continuous or the  problem is a jointly convex problem, besides the implicit assumption that the maximal value function is locally Lipschitz continuous. In this subsection, we do not employ these conditions and directly investigate the subdifferentials of the maximal value function.

The following example illustrates that the assumptions in Subsection \ref{subsec} fail but all the required conditions in this subsection hold.
\begin{ex}
Consider the following simple example:
\begin{eqnarray*}
\max_y \ y(x+1)\quad {\rm s.t.} \ -5\le x+y \le 0.
\end{eqnarray*}
By a simple calculation, the optimal solution $\cS(x)$ is given by
\begin{eqnarray*}\label{ex-map}
\cS(x) = \left\{\begin{array}{cc} \{-x\} & {\rm if}\ x > -1, \\ [4pt] [-4,1] & {\rm if}\  x = -1, \\[4pt] \{-x-5\} & {\rm if}\ x < -1. \end{array}\right.
\end{eqnarray*}
It is easy to verify that $\cS(x)$  is not inner semi-continuous at any $(\bx,\by)$ with $\bar x=-1,\bar y \in (-4,1)$ and the problem is not jointly convex.
\end{ex}

To give the main result of this subsection, we assume that $-f(x,y)$, $g_i(x,y)\ i=1,\ldots,q $ are second order continuously differentiable in $(x,y)$,  and are convex in $y$ for all $x$ near a reference point $\bar{x}$. We now introduce a dual of problem $(P_\bx)$ proposed in Wolfe
 \cite{wolfe} as follows:
\begin{eqnarray*}
(D_\bx)~~~~~~\min_{y,\lambda} && \cL(\bx,y,\lambda) \\
{\rm s.t.}  && \nabla_y \cL(\bx,y,\lambda) = 0,\ \lambda\ge0.
\end{eqnarray*}
We denote by ${\Omega}(\bx)$ the feasible region and
\begin{equation}\label{ValuefunctionD}
V_D(\bx):=\inf_{y,\lambda}\{ \cL(\bx,y,\lambda): (y,\lambda)\in \Omega(\bx)\}
\end{equation}  the optimal  value function of problem $(D_\bx)$.

\begin{lema}(Wolfe \cite[Theorem 1]{wolfe})\label{weakdulity}
Suppose that $f(x,y)$ is concave in variable  $y$, and $g_i(x,y)\ i=1,\ldots,q $ are convex in variable  $y$ for  {any given $x
$.} Then the weak duality holds, that is,  we have $V(x) \le V_D(x)$.
\end{lema}

For a given scalar $r\in \{0,1\}$ and  $(y,\lambda)\in \Omega(\bx)$, we denote the set of generalized multipliers for problem $(D_{\bx})$ at $(y,{\lambda})$ as:
\[\Xi^r(\bx, y, {\lambda}):=
\left\{u: \nabla_{yy}^2 \cL(\bx,y,{\lambda}) u =0,\
 0\le -\nabla_{y}g(\bx,y)^\top u- r g(\bx,y)\ \bot\ {\lambda} \ge0\right \}.
\]
We give some comments on the set $\Xi^r(\bx,y,{\lambda})$. First note that we always have $0\in \Xi^r(\bx,y,{\lambda})$.  MFCQ holds at $(y,\lambda)\in \Omega(\bar x)$ for problem $(D_\bx)$ if $\Xi^0(\bx,y,{\lambda})=\{0\}.$ Moreover, $\Xi^0(\bx,y,{\lambda})\subseteq \Xi^1(\bx,y,{\lambda})$ since $g(\bx,y)\le0$.

{Using the weak duality, we can study the value function for a maximization problem as a value function  for a minimization problem. Hence in the following theorem, we can obtain upper estimates of the  subdifferentials for the  maximal value function without using the convex hull on all  optimal solutions.}

For minimization problems, a so-called restricted inf-compactness condition  firstly introduced in Clarke \cite[Hypothesis 6.5.1]{Clarke} and termed in Guo et al. \cite[Definition 3.8]{sensitivity}, is used to ensure the lower semi-continuity of the value function. For maximization problems, we propose the following restricted sup-compactness condition.
\begin{defi}[Restricted sup-compactness]\label{Defn3.2}
We say that the restricted sup-compactness condition holds at $\bar x$  if $V(\bar x)$ is finite and there exist  positive scalars $\epsilon, \delta$ and a compact set $\Omega$ such that for all $ x\in \cB_\delta(\bar x)$ with $V(x) > V(\bar x)-\epsilon$, one always has $\cS(x)\cap \Omega \neq\emptyset$.
\end{defi}
It is easy to see that the restricted sup-compactness condition is weaker than the sup-compactness condition, i.e., there exist $\delta>0,\alpha<V(\bar x)$, and
a compact set $\Omega$ such that 
$$
\{ y\in {\cal F}{(x)}: f(x, y) \geq \alpha, x \in {\cal B}_\delta(\bar x)\subseteq \Omega.
$$ 
Moreover it is easy to see that the sup-compactness condition is weaker than the uniform compactness of the feasible solution mapping ${\cal F}$, i.e., the closure of $\bigcup_{x\in V}\cF(x)$ is compact for some neighborhood $V$ of $\bar x$.

\begin{thm}\label{thm sub}
Let $\by\in \cS(\bx)$ and  $\bar{\lambda}\in \Sigma(\bx,\by)$. Suppose that a CQ holds at  $(\bx,\by,\bar{\lambda})$   for the following constraint system in $(x,y,\lambda)$:
\begin{equation}\label{constraintsystem}
\nabla_y \cL(x,y,\lambda) = 0,\ \lambda\ge0.
\end{equation}
Then we have
\begin{eqnarray}
\label{regular}
&& \widehat{\partial}V(\bx) \subseteq \left\{\nabla_x \cL(\bx,\by,\bar\lambda)+ \nabla_{xy}^2\cL(\bx,\by,\bar\lambda) u:
u\in \Xi^1(\bx,\by,\bar{\lambda})\right\}.
 \end{eqnarray}
Suppose that the restricted sup-compactness holds at $\bx$,  a stable parametric CQ holds for $(P_\bx)$ at all $y\in \cS(\bx)$, and {a stable CQ} holds at all $(\bx,  \tilde y,\tilde{\lambda})$ for the constraint system (\ref{constraintsystem}) where $\tilde y\in {\cal S}(\bx)$ and $\tilde\lambda\in \Sigma(\bx,\tilde y)$.
Then we have
\begin{eqnarray}
&& \partial V(\bx) \subseteq \bigcup_{{y\in {\cal S}(\bar x)}} \left\{\nabla_x \cL(\bx,y, \lambda) + \nabla_{xy}^2 \cL(\bx,y, \lambda) u :
u\in \Xi^1(\bx,y, {\lambda}), \lambda \in \Sigma(\bar x,y)
\right\},\label{limiting}\\
&& \partial^\infty V(\bx) \subseteq \bigcup_{y\in {\cal S}(\bar x)} \left\{\nabla_{xy}^2 \cL(\bx,y,\lambda) u:u\in \Xi^0(\bx,y, {\lambda}) , \lambda \in \Sigma(\bar x,y)
\right\}.\label{horizon}
\end{eqnarray}
\end{thm}
\begin{proof}
(i) Let $\xi\in \widehat{\partial} V(\bx)$. Then by the definition of regular subdifferentials, for any $\epsilon>0$, there exists $\delta_\epsilon>0$ such that
\[
V(x) - V(\bx) - \xi^\top (x-\bx) \ge -\epsilon \|x-\bx\|,\quad \forall x\in \cB_{\delta_\epsilon}(\bx).
\]
Then by the weak duality in Lemma \ref{weakdulity},  we have
$$
V_D(x) - \xi^\top (x-\bx) +\epsilon \|x-\bx\| \ge f(\bx,\by),\quad \forall x\in \cB_{\delta_\epsilon}(\bx).
$$
Thus by the definition of $V_D(x)$, it follows that
\begin{equation}\label{inequ1}
\cL(x,y,\lambda)  -\xi^\top (x-\bx)  + \epsilon \|x-\bx\| \ge f(\bx,\by),\quad \forall x\in \cB_{\delta_\epsilon}(\bx),\ \forall (y,\lambda)\in \Omega(x).
\end{equation}
Since $(\by,\bar\lambda)\in \Omega(\bx)$ by $\bar{\lambda}\in \Sigma(\bx,\by)$,  the relation \eqref{inequ1} implies that $(\bx,\by,\bar{\lambda})$  is a local minimizer of the problem
\begin{eqnarray}\label{dualmin}
\begin{array}{rl}
\min\limits_{x,y,\lambda} & \cL(x,y,\lambda) -\xi^\top (x-\bx) +\epsilon \|x-\bx\|\\[5pt]
{\rm s.t.} & \nabla_y \cL(x,y,\lambda) = 0,\ \lambda\ge0.
\end{array}
\end{eqnarray}
Since a CQ holds at $(\bx,\by,\bar{\lambda})$, it follows that $(\bx,\by,\bar{\lambda})$ is a KKT point of problem \eqref{dualmin}, i.e., there exists a multiplier $u$ such that
\begin{eqnarray*}
&&\xi \in \nabla_x \cL(\bx,\by,\bar{\lambda}) + \epsilon \mathbb{B} + \nabla_{xy}^2 \cL(\bx,\by,\bar{\lambda}) u,\\
&&\nabla_{yy}^2\cL(\bx,\by,\bar{\lambda})u = 0,\ 0\le -\nabla_{y}g(\bx,\by)^\top u-g(\bx,\by) \ \bot\ \bar{\lambda} \ge0.
\end{eqnarray*}
The formula \eqref{regular} follows from the definition of $\Xi^1(\bx,\by,\bar{\lambda})$ and the arbitrariness of  $\epsilon$.

(ii) Let $\xi\in \partial V(\bx)$. Then by the definition of limiting subdifferentials, there exist sequences $x^k\to\bx$ with $V(x^k)\to V(\bx)$ and $\xi^k \in \widehat{\partial} V(x^k)$ such that $\xi^k\to \xi$. Since $V(x^k)\to V(\bx)$, by the restricted sup-compactness assumption, there exists a bounded sequence $\{y^k\in \cS(x^k)\}$. Assume, without loss of generality, that $y^k\to {\tilde y}$ as $k\to\infty$. Thus by the continuity of $f(x,y)$, it follows that
\[
V(\bx) = \lim_{k\to\infty} V(x^k) = \lim_{k\to\infty} f(x^k,y^k) =f(\bx,\tilde y).
\]
Thus we have $\tilde y\in \cS(\bx)$.

Since a stable parametric CQ holds at $\tilde y\in \cS(\bx)$ for problem $(P_\bx)$, this CQ also holds at $y^k\in \cS(x^k)$ for all $k$ sufficiently large for problem $(P_{x^k})$. Thus for all $k$ sufficiently large, there exists KKT multiplier $\lambda^k$ such that
\[
\nabla_y f(x^k,y^k) - \nabla_y g (x^k,y^k) \lambda^k=0, \ 0\le -g (x^k,y^k)\ \bot \ \lambda^k\ge0.
\]
This means that
$$\lambda^k \in \cM(\nabla_y f(x^k,y^k) ,x^k, y^k)
,$$
where $\cM(\cdot)$  is the multiplier set defined as in (\ref{stand multi}).   By the definition of the stable parametric CQs, we may find a subsequence on
which $\{\lambda^k\}$ converges to some {$\tilde {\lambda}$} with $\lambda^k \in \Sigma(x^k,y^k)$ and $\tilde {\lambda}\in \Sigma(\bx,\tilde{y})$. Without loss of generality, we assume that the whole sequence $\{(y^k,\lambda^k)\}$ converges to $(\tilde y,\tilde {\lambda})$ as $k\to\infty$.

Since a stable CQ holds at $(\bx, \tilde y,\tilde{\lambda})$ for the constraint system \eqref{constraintsystem} and $(x^k, y^k,\lambda^k) \to (\bx, \tilde y,\tilde{\lambda})$ as $k\to\infty$, by the local persistence, it follows that for all $k$ sufficiently large, this CQ holds at $(x^k, y^k,\lambda^k)$
for the constraint system (\ref{constraintsystem}).
Thus from the result (i), we can see that \eqref{regular} holds if one replaces $(\bar x,\bar y,\bar \lambda)$  with $(x^k, y^k,\lambda^k)$. Then by \eqref{regular}, the relation $\xi^k \in \widehat{\partial} V(x^k)$ implies that there exists $u^k$ such that
\begin{itemize}
\item[\rm (a)] $\xi^k = \nabla_x \cL(x^k,y^k,\lambda^k) + \nabla_{xy}^2 \cL(x^k,y^k,\lambda^k)u^{k}$,
\item[\rm (b)] $\nabla_{yy}^2 \cL(x^k,y^k,\lambda^k)u^{k}=0$,
\item[\rm (c)] $0\le -\nabla_{y}g(x^k,y^k)^\top u^k-g(x^k,y^k)\ \bot\ \lambda^k \ge0$.
\end{itemize}
By the continuity of functions $\nabla_x \cL(x,y,\lambda)$ and $g(x,y)$, it follows that as $k\to\infty$,
\begin{eqnarray}
&& \xi^k - \nabla_x \cL(x^k,y^k,\lambda^k) \to \xi -\nabla_x \cL(\bx,\tilde y, \tilde{\lambda}),\label{limit3} \\[4pt]
&& g(x^k,y^k)\to g(\bx,\by). \label{limit5}
\end{eqnarray}
Recall that a stable CQ holds for the constraint system (\ref{constraintsystem}) at $(\bx,\tilde y,\tilde{\lambda})$. Then by condition (ii) of stable CQs, from the relations (a)-(c) and \eqref{limit3}-\eqref{limit5} it follows that there exists $u$ such that
\begin{itemize}
\item $\xi = \nabla_x \cL (\bx,\tilde y,\tilde {\lambda}) + \nabla_{xy}^2 \cL(\bx,\tilde y,\tilde {\lambda})u$,
\item $\nabla_{yy}^2\cL(\bx,\tilde y,\tilde {\lambda})u = 0,\ 0\le -\nabla_{y}g(\bx,\tilde y)^\top u-g(\bx,\tilde y)\ \bot\ \tilde{\lambda} \ge0$.
\end{itemize}
Then the formula \eqref{limiting} follows immediately from the above results and the definition of $\Xi^1(\bx,\tilde y,\tilde{\lambda})$.

(iii) Let $\xi\in \partial^\infty V(\bx)$. Then by the definition of horizon subdifferentials, there exist $x^k\to\bx$ with $V(x^k)\to V(\bx)$, $t_k\to0^+$,  and $\xi^k \in \widehat{\partial} V(x^k)$ such that $t_k\xi^k\to \xi$. Similarly as the proof for the result (ii), it follows from the relation $\xi^k \in \widehat{\partial}V(x^k)$ that for all $k$ sufficiently large, there exists $u_k$ such that (a)-(c) in the proof of the result (ii) hold. Multiplying (a)-(c)  by $t_k$ yields
\begin{itemize}
\item[(a1)] $t_k\xi^k = t_k\nabla_x \cL(x^k,y^k,\lambda^k) + \nabla_{xy}^2\cL(x^k,y^k,\lambda^k) (t_ku^k)$,
\item[(b1)] $ \nabla_{yy}^2\cL(x^k,y^k,\lambda^k) (t_ku^k) = 0$,
\item[(c1)] $0\le -\nabla_{y}g(x^k,y^k)^\top (t_ku^k)-g(x^k,y^k) t_k \ \bot\ \lambda^k \ge0$.
\end{itemize}
By the continuity of functions $\nabla_x \cL(x,y,\lambda)$ and $g(x,y)$, and the facts that $t_k\xi^k\to \xi$ and $t_k\to0^+$, it follows that as $k\to\infty$,
\begin{eqnarray}
&& t_k\xi^k - t_k\nabla_x \cL(x^k,y^k,\lambda^k) \to \xi, \ g(x^k,y^k) t_k\to 0.\label{limit2}
\end{eqnarray}
Note that a stable CQ holds for the constraint system \eqref{constraintsystem} at $(\bx,\tilde y,\tilde {\lambda})$. Then by condition (ii) of stable CQs, using the relations (a1)-(c1) and \eqref{limit2} implies that there exists $u$ such that
\begin{itemize}
\item $\xi =  \nabla_{xy}^2 \cL(\bx,\tilde y,\tilde {\lambda}) u$,
\item $\nabla_{yy}^2 \cL(\bx,\tilde y,\tilde{\lambda}) u =0$,\ $0\le -\nabla_{y}g(\bx,\tilde y)^\top u\ \bot\ \tilde{\lambda} \ge0$.
\end{itemize}
Then the formula \eqref{horizon} follows immediately from the above results and the definition of $\Xi^0(\bx,\tilde y,\tilde{\lambda})$.
\end{proof}

By Clarke \cite[Corollary 1 to Theorem 6.5.2]{Clarke}, under the restricted sup-compactness at $\bx$ and  MFCQ at all  $y\in {\cal S}(\bar x)$, the value function $V(x)$ is Lipschitz continuous at $\bx$. The following corollary provides an alternative sufficient condition for Lipschitz continuity of the value function.


\begin{cor}\label{Cor4.2}
Suppose that the restricted sup-compactness holds at $\bx$ and
a stable parametric CQ holds at all $y\in \cS(\bx)$ for problem $(P_\bx)$,
  {a stable CQ} holds at all $(\bx,\tilde y,\tilde{\lambda})$  for the constraint system (\ref{constraintsystem}) where  $\tilde y \in \cS(\bx)$ and $\tilde{\lambda}\in \Sigma(\bx,\by)$.  Moreover assume that the  value function $V(x)$ is continuous at $\bar x$.
 If
 \begin{equation}\label{condition2}
\bigcup_{y\in {\cal S}(\bar x)} \left\{\nabla_{xy}^2 \cL(\bx,y,\lambda) u:u\in \Xi^0(\bx,y, {\lambda}) , \lambda \in \Sigma(\bar x,y)
\right\} =\{0\},
 \end{equation}
then $V(x)$ is Lipschitz continuous at $\bx$  and the upper estimate (\ref{limiting}) holds.
In addition to the above assumptions, if
$$ \bigcup_{{y\in {\cal S}(\bar x)}} \left\{\nabla_x \cL(\bx,y, \lambda) + \nabla_{xy}^2 \cL(\bx,y, \lambda) u :
u\in \Xi^1(\bx,y, {\lambda}), \lambda \in \Sigma(\bar x,y)
\right\} =\{\zeta\},$$
then $V(x)$ is strictly differentiable and $\nabla V(\bar x)=\zeta$.
\end{cor}
\begin{proof}
Since $V(x)$ is continuous at $\bx$,  we have  $0 \in \partial^\infty V(\bx)$ (Rockafellar and Wets \cite[Theorem 8.9]{Rock98}). By Theorem \ref{thm sub} and condition \eqref{condition2}, it follows that $\partial^\infty V(\bx) \subseteq \{0\}$. Thus, $\partial^\infty V(\bx) = \{0\}$ and hence $V(x)$ is Lipschitz continuous at $\bx$ (Rockafellar and Wets\cite[Theorem 9.13]{Rock98}). The rest of the proof follows from Theorem \ref{thm sub} immediately.
\end{proof}

In Corollary \ref{Cor4.2},  the continuity of the value function $V(x)$ is needed.  The following proposition gives a very weak sufficient condition for the continuity of the value function.
In fact the condition is even weaker than  RS at $\bar x$ which is weaker than MFCQ at one solution $\bar y \in {\cal F}(\bar x)$ which, together with the uniform compactness of the  feasible map $\cF(x)$ around $\bar x$, was shown to be sufficient for ensuring the continuity by Gauvin and Dubeau in \cite[Theorems 3.3 and 5.1]{Gauvin-Dubeau}.
\begin{prop}\label{lema-conti}
Consider the parametric optimization problem $(P_x)$. Suppose that the restricted sup-compactness holds at $\bar x$ and there exists $\bar{y}\in S(\bar{x})$ such that $\lim\limits_{x\to\bar{x}} d(\bar{y},\cF(x))=0$. Then $V(x)$ is continuous at $\bx$.
\end{prop}
\begin{proof} Under the restricted sup-compactness at $\bar x$, it is easy to verify that $V(x)$ is upper semi-continuous at $\bar x$; see a similar argument in the proof of Guo et al.\cite[Theorem 3.9]{sensitivity}.  We now prove the lower semi-continuity of $V(x)$ at $\bx$. Let $\{x_k\}$ be a sequence such that $x_k\to \bar{x}$ and {$V(x_k)\to \displaystyle \liminf_{x\to \bar{x}}
V(x).$} Then $\displaystyle \lim_{k\to\infty} {\rm dist}(\bar{y},\cF(x_k))=0$. For each $k$, take $y_k\in \cF(x_k)$ such that
$\|y_k-\bar{y}_k\| < {\rm dist}(\bar{y},\cF(x_k)) + 1/k$. It follows that $(x_k; y_k) \to (\bar{x};\bar{y})$ and so
\[
V(\bar{x}) = f(\bar{x}; \bar{y}) = \lim_{k\to\infty}f(x_k; y_k) \le \lim_{k\to\infty}V(x_k)
= \liminf_{x\to \bar{x}}V(x)
\]
This shows that $V$ is lower semi-continuous at $\bar{x}$. The proof is complete.
\end{proof}

In Theorem \ref{thm sub}, the upper estimates depend on the set $\Xi^r(\bx,y,{\lambda})$ with $r=\{0,1\}$.
We now explore sufficient conditions under which $\Xi^0(\bx,y,{\lambda})=\Xi^1(\bx,y,{\lambda})= \{0\}$. Let $y \in \cS(\bx)$ and ${\lambda}\in \Sigma(\bx,y)$. Recall that the critical cone $\cC(\bx,y)$ of problem $(P_\bx)$ can be written as
\[
\left  \{d: \begin{array}{l}\nabla_y g_i(\bx,y)^\top d =0,\ {\rm if}\ i\in\cI_g(\bx,y),  \lambda_i >0\\[4pt] \nabla_y g_i(\bx,y)^\top d \le 0,\ {\rm if}\ i\in\cI_g(\bx,y),  \lambda_i =0 \end{array}\right \},
\]
and the  second order sufficient condition (SOSC) holds at $( y,{\lambda})$ for problem $(P_\bx)$ if
\[
\langle d, \nabla^2_{yy} \cL(\bx,y,\lambda)d\rangle < 0,\quad \forall d\in \cC(\bx,y)\backslash \{0\}.
\]

\begin{prop}\label{pro-empty}
Let $y\in \cS(\bx)$ and ${\lambda}\in \Sigma(\bx,y)$. If SOSC holds at $( y,{\lambda})$ for problem $(P_\bx)$, then $\Xi^0(\bx,y,{\lambda})= \Xi^1(\bx,y,{\lambda})=\{0\}$.
\end{prop}
\begin{proof}
To the contrary, assume that $0\neq u\in \Xi^1(\bx,y,{\lambda})$. Then we have
\begin{eqnarray}
&& \nabla_{yy}^2 \cL(\bx,y,{\lambda}) u =0,\label{p1}\\
&& 0\le -\nabla_{y}g(\bx,y)^\top u-g(\bx,y)\ \bot\ {\lambda} \ge0.\label{p2}
\end{eqnarray}
By \eqref{p2}, it follows that for $i=1,\ldots,m$,
$$
{\lambda}_i\ge 0, -\nabla_{y}g_i(\bx,y)^\top u+g_i(\bx,y)\ge 0, {\lambda}_i(\nabla_{y}g_i(\bx,y)^\top u-g_i(\bx,y))=0.
$$
If $i\in \cI_g(\bx,y)$ and ${\lambda}_i>0$, one can easily have $\nabla_{y}g_i(\bx,y)^\top u=0$. If $i\in \cI_g(\bx,y)$ and ${\lambda}_i=0$, one can easily have $\nabla_{y}g_i(\bx,y)^\top u\le0$. Thus, $u\in \cC(\bx,y)$. Then the validity of SOSC implies that
\[\langle u, \nabla^2_{yy} \cL(\bx,y,\lambda)u\rangle <0, \]
contradicting \eqref{p1}. Thus, $\Xi^1(\bx,y,{\lambda})=\{0\}$. Since $\Xi^0(\bx,y,{\lambda}) \subseteq \Xi^1(\bx,y,{\lambda})$, one can easily have that $\Xi^0(\bx,y,{\lambda})=\{0\}$.
\end{proof}

We now show if either $f(x,y)$ and $g(x,y)$ are separable with respect to $x$ and $y$ or SOSC holds, then the upper estimates (\ref{limiting})-(\ref{horizon}) are much simplier. In particular, no second order derivatives are involved.
\begin{cor}\label{cor-regu}
Let $\by\in \cS(\bx)$ and  $\bar{\lambda}\in \Sigma(\bx,\by)$. Suppose that either $f(x,y)$ and $g(x,y)$  are separable with respect to $x$ and $y$ and a CQ holds at  $(\bx,\by,\bar{\lambda})$ for the constraint system (\ref{constraintsystem}) or SOSC holds at $(\bar y,\bar\lambda)$ for problem $(P_\bx)$. Then we have
\begin{equation*}
\widehat{\partial} V(\bx) \subseteq \left\{\nabla_x \cL(\bx,\by,\bar\lambda)\right\}.
 \end{equation*}
\end{cor}
\begin{proof}  If $f(x,y)$ and $g(x,y)$ are separable with respect to $x$ and $y$, then $\nabla_{xy}^2 \cL(\bx,y,{\lambda})u=0$ for all $u$. If SOSC holds at $(\bar y,\bar{\lambda})$ for problem $(P_\bx)$, then by Proposition \ref{pro-empty}, $\Xi^0(\bx,\bar y,\bar{\lambda})=\Xi^1(\bx,\bar y,\bar{\lambda})=\{0\}.$ This implies that  MFCQ holds for problem $(D_\bx)$  at $(\bar y,\bar \lambda) \in \Omega(\bar x) $ and thus  MFCQ holds at $(\bx,\tilde y,\tilde{\lambda})$ for system \eqref{constraintsystem}. The desired inclusion follows immediately from (\ref{regular}) of Theorem \ref{thm sub}.
\end{proof}

\begin{cor}\label{cor2}
Assume that the restricted sup-compactness holds for $(P_\bx)$ and a stable parametric CQ holds for $(P_\bx)$ at all $y\in \cS(\bx)$. Further suppose that either $f(x,y)$ and $g(x,y)$ are separable with respect to $x$ and $y$ and  {a stable CQ} holds at all $(\bx,\tilde y,\tilde{\lambda})$ for the constraint system (\ref{constraintsystem}) or SOSC holds at all $(\tilde y,\tilde{\lambda})$ for $(P_\bx)$ where $ \tilde y\in \cS(\bx)$ and $\tilde{\lambda}\in \Sigma(\bx,\tilde y)$.
Then
\begin{eqnarray*}\label{sub-diff}
 \partial V(\bx) \subseteq \bigcup_{ y\in \cS(\bx)} \left\{\nabla_x \cL(\bx, y,\lambda):\lambda \in \Sigma(\bx,y)\right\},\
\partial^\infty V(\bx) {\subseteq} \{0\}.
\end{eqnarray*}
If, in addition, the RS holds for $\cF(x)$ at some $\by \in \cS(\bx)$, then $V(x)$ is Lipschitz continuous at $\bx$.
\end{cor}
\begin{proof}
The result for the first part follows immediately from (\ref{limiting})-(\ref{horizon}) by using a similar argument as in the proof of Corollary \ref{cor-regu}. Suppose that the RS holds for $\cF(x)$ at some $\by \in \cS(\bx)$. Then by Proposition \ref{lema-conti}, the value function $V(x)$ is continuous at $\bx$. Thus, $0 \in \partial^\infty V(\bx)$ (Rockafellar and Wets \cite[Theorem 8.9]{Rock98}). Then we can have $\partial^\infty V(\bx) = \{0\}$ and hence $V(x)$ is  Lipschitz continuous at $\bx$ (Rockafellar and Wets \cite[Theorem 9.13]{Rock98}).
\end{proof}

It is interesting to note that  our result above may apply to the case where  MFCQ fails at certain optimal solution so that Gauvin and Dubeau \cite[Theorem 5.3]{Gauvin-Dubeau} may not be applicable. Moreover, we obtain the upper estimates for the limiting subdifferential from which we can obtain the one for the Clarke subdifferential. But Gauvin and Dubeau \cite[Theorem 5.3]{Gauvin-Dubeau} only provide the upper estimates for the Clarke subdifferential from which we cannot obtain the upper estimates for the limiting subdifferential.

In the last part of this section, we  relate the subdifferential of the maximal value function  to the multiplier set in the case when the perturbation is canonical. Consider the following nonparametric maximization problem
\begin{equation}\label{mxp}
\max_y\quad  c(y) \quad {\rm s.t.}\quad d(y) \le 0,
\end{equation}
where $c:\mathbb{R}^m \rightarrow \mathbb{R},d:\mathbb{R}^m \rightarrow \mathbb{R}^p$ are second order continuously differentiable functions. We assume that an optimal solution exists.
In economics, the maximization problem (\ref{mxp}) can be used to model the problem of maximizing production under a budget constraint; see e.g., Intriligator \cite[page 149]{Intri}.  Let $U(x):=\max\{c(y): d(y)-x\le 0\}$ be the maximal value function. In economic theory, it is well-known that if the maximum value function is continuously differentiable and the multiplier is unique, then the gradient of the maximum value function is equal to the multiplier. Since the gradient of the maximum value function is the change rate of the maximum value subject to the  change in the right hand side of the constraint, the multiplier can be interpreted as a shadow price. The following corollary is an extension of this result to more general cases where the maximum value function may not be smooth and the multiplier set may not be a singleton.

\begin{cor}
Consider problem \eqref{mxp}.
Let $\by$ be an optimal solution of problem \eqref{mxp} and $\Sigma(\bar y)$ be the associated multiplier set at $\by$. Then the maximal value function $U(x)$ is Lipschitz continuous at $\bar x=0$ and
\begin{equation*}
\partial U (0) \subseteq \Sigma(\bar y)
\end{equation*} under one of the following assumptions:
\begin{itemize}
\item[1)] The restricted sup-compactness holds at $\bar x=0$ and a stable CQ  holds at all solutions for problem \eqref{mxp}  (e.g., one of LICQ, MFCQ, CRCQ, RCRCQ, RCPLD,  the quasi-normality holds), and a stable  CQ holds at all $(\tilde y,\tilde \lambda)$   for the system
 \begin{equation*}\label{system}
 \nabla c(y)-\nabla d(y)\lambda=0,\ \lambda\geq 0,\end{equation*}
 where $\tilde y$ is an optimal solution and $\tilde \lambda \in \Sigma(\tilde y).$
 \item[2)] The function $c(y)$ is {strongly} concave and $d(y)$ is convex. A stable CQ  for problem \eqref{mxp} holds at the unique optimal solution $\bar y$.
 \item[3)] The function $c(y)$ is {strongly} concave and $d(y)$ is linear.
\end{itemize}
\end{cor}
\begin{proof}
1) Since $f(x,y):=c(y)$ and $g(x,y):=d(y)-x$ are separable in variables $x$ and $y$, the results follow from Corollary \ref{cor2}.

2) Since $c(y)$ is strongly concave and $d(y)$ is convex,  the restricted sup-compactness holds and SOSC holds. Moreover, $\by$ is the unique solution. The result follows from Corollary \ref{cor2}.

3) Since $d(y)$ is linear, the stable CQ required in the result 2) holds automatically. Thus the desired result follows from the result 2).
\end{proof}

\section{Necessary optimality conditions for minimax problems}\label{sec:minimax}

In this section, we develop optimality conditions for nonconvex minimax problems in the form of
\begin{equation}
(P_{\rm minimax})~~~~~~~ \min_{x\in X}\max_{y\in \cF(x)}f(x,y),\quad \cF(x):=\{y:g(x,y)\leq 0\},\label{minmaxg}
\end{equation}
where   $f:\Re^{n+m}\to \Re, g:\Re^{n+m}\to \Re^q$ and $X$ is a closed subset of $\Re^n$. Unless otherwise specified, we assume that $f(x,y),g(x,y)$ are continuously differentiable functions.

\subsection{Optimality conditions for nonconvex-concave minimax problems by MPEC approach}
In this subsection we review  MPEC approach for solving the nonconvex-concave minimax problem (see e.g., Stein \cite{Stein}).
Suppose that  $f(x,y),g(x,y)$ are second order continuously differentiable, and $-f(x,y),g(x,y)$ are convex in variable $y$ for all $x\in X$. Let $X:=\{x:h(x)\le 0\}$ with a continuously differentiable function $h:\Re^n\to \Re^r$. Replacing the inner maximization problem  by its KKT condition results in the following MPEC reformulation:
\begin{eqnarray*}
{\rm (P_{MPEC})}~~~~~~\min_{x,y, \lambda } && f(x,y)\\
\mbox{ s.t. } && x\in X,\ \nabla_y \cL(x,y,\lambda) = 0,\\
              && 0\leq -g(x,y) \perp \lambda \geq 0.
\end{eqnarray*}

We observe that the objective function of ${\rm (P_{MPEC})}$ is independent of the variable $\lambda$. Then the equivalence between ${\rm (P_{minimax})}$ and ${\rm (P_{MPEC})}$ in the sense of global minimizers is easily derived under very mild conditions (Dempe and Dutta \cite{Dem-Dut}) since any global minimizer of ${\rm (P_{MPEC})}$ corresponds to that of ${\rm (P_{minimax})}$. However, local minimizers of ${\rm (P_{MPEC})}$ may not correspond to those of the original minimax problem due to the introduction of extra multiplier as an implicit variable. By using the proof techniques for Theorems 2.1 and 3.2 in  Dempe and Dutta \cite{Dem-Dut}, we can have the following result.

\begin{thm}[Equivalence theorem]\label{thm-equ}
Let $(\bx,\by)$ be a locally optimal solution  of ${\rm (P_{minimax})}$ and a CQ hold at $\by\in \cF(\bx)$. Then for all $\bar\lambda\in \Sigma(\bx, \by)$, the point $(\bx,\by,\bar\lambda)$ is a locally optimal solution of ${\rm (P_{MPEC})}$. On the other hand, let a stable parametric CQ hold at $(\bx,\by)$ and $(\bx,\by,\bar\lambda)$ be a locally optimal solution of ${\rm (P_{MPEC})}$ for all $\bar\lambda\in \Sigma(\bx, \by)$. Then $(\bx, \by)$ is a locally optimal solution  of ${\rm (P_{minimax})}$.
\end{thm}

By using the stationarity concepts for MPECs in the literature (e.g., \cite{Scheel-Scholtes,Ye-jmaa}), the stationarity conditions for  ${\rm (P_{MPEC})}$ can be derived.  For any feasible point $(\bx,\by,\bar\lambda)$ of ${\rm (P_{MPEC})}$, we define the following index sets.
\begin{eqnarray*}
&& \cI_{0+}=\cI_{0+}(\bx,\by,\bar\lambda):=\{i: -g_i(\bx,\by)=0, \bar\lambda_i >0\},\\
&& \cI_{+0}=\cI_{+0}(\bx,\by,\bar\lambda):=\{i: -g_i(\bx,\by)>0, \bar\lambda_i =0\},\\
&& \cI_{00}=\cI_{00}(\bx,\by,\bar\lambda):=\{i: -g_i(\bx,\by)=0, \bar\lambda_i =0\}.
\end{eqnarray*}

\begin{defi}[Stationarity conditions]\label{defi-station}
We say that $(\bx,\by,\bar\lambda)$ is a weakly stationary point of ${\rm (P_{MPEC})}$ if there exist multipliers $(\bar u,\bar\alpha,\bar\beta)\in \Re^m\times \Re^q \times \Re^q$ such that
\begin{eqnarray}\label{weak-sta}
&& 0\in \nabla_xf(\bx,\by) + \nabla_x g(\bx,\by)\bar \alpha +\nabla_{xy}^2 \cL(\bx,\by,\bar\lambda)\bar u +\cN_X(\bx),\nonumber \\
&& \nabla_yf(\bx,\by) + \nabla_y g(\bx,\by)\bar \alpha + \nabla_{yy}^2 \cL(\bx,\by,\bar\lambda) \bar u =0,\nonumber \\
&& -\nabla_y g(\bx,\by)^\top \bar u -\bar \beta=0,\\
&& \nabla_y \cL(\bx,\by,\bar\lambda) =0,\ 0\leq -g(\bx,\by) \perp \bar \lambda \geq 0,\nonumber \\
&& \bar \alpha_i = 0\ \ i\in \cI_{+0},\  \bar \beta_i = 0\ \ i\in \cI_{0+}. \nonumber
\end{eqnarray}
We say that $(\bx,\by,\bar\lambda)$ is a C-stationary point if \eqref{weak-sta} holds and \[
\bar \alpha_i \bar\beta_i \ge0\  \forall i\in \cI_{00}.
\]
We say that $(\bx,\by,\bar\lambda)$ is an M-stationary point if \eqref{weak-sta} holds and
\[
{\rm either}\ (\bar\alpha_i,\bar\beta_i)>0\ {\rm or}\  \bar\alpha_i\bar\beta_i =0\ \forall i\in \cI_{00}.
\]
We say that $(\bx,\by,\bar\lambda)$ is an S-stationary point if \eqref{weak-sta} holds and
\begin{equation}\label{s-multiplier}
(\bar\alpha_i,\bar\beta_i)\ge0\ \forall i\in \cI_{00}.
\end{equation}
\end{defi}

It is easy to see that S-stationarity is  strongest among the stationarity conditions in Definition \ref{defi-station}. Since local minimizers of ${\rm (P_{MPEC})}$ can be S-stationary under an MPEC-LICQ condition (Scheel and Scholtes \cite[Theorem 2]{Scheel-Scholtes}), by Theorem \ref{thm-equ} the following result follows immediately.

\begin{thm}\label{thm-stampec}
Let $(\bx,\by)$ be a locally optimal solution  of ${\rm (P_{minimax})}$ and $\bar\lambda\in \Sigma(\bx, \by)$. If MPEC-LICQ holds at $(\bx,\by,\bar\lambda)$, i.e., LICQ holds at $(\bx,\by,\bar\lambda)$ for the following system:
\begin{eqnarray}\label{cons-licq}
h(x)\le 0, \ \nabla_y \cL(x,y,\lambda) = 0, g_i(x,y)= 0\ i\in \cI_{0+}\cup \cI_{00}, \lambda_i = 0\ i\in \cI_{+0}\cup \cI_{00},
\end{eqnarray}
then the point $(\bx,\by,\bar\lambda)$ is an S-stationary point of ${\rm (P_{MPEC})}$.
\end{thm}

We point out that MPEC-LICQ is a very stringent condition. It must fail if the multiplier set $\Sigma(\bx,\by)$ is not a singleton and the constraint function $g(x,y)=g(y)$ is independent of $x$; see more discussions in Gfrerer and Ye \cite{HY2018}. Indeed, when $\Sigma(\bx,\by)$ is not a singleton, the family of gradients $\{\nabla g_i(\by): i\in \cI_g(\by)\}$ is linearly dependent, resulting in that the family of gradients of the functions in the system \eqref{cons-licq} is linearly dependent. Thus, LICQ fails for the system \eqref{cons-licq}.
If we would like to find a weaker M-stationary or  C-stationary point, MPEC-LICQ can be weakened to be weaker MPEC-tailored CQs (e.g., Ye \cite{Ye-jmaa}). We omit the detailed discussions on these weaker stationarity conditions since we will compare the strongest S-stationary condition with that derived by employing the subdifferential of maximal value function in Subsection \ref{sub-opti}.

\subsection{Optimality conditions for nonconvex-concave minimax problems by Wolfe duality approach}\label{sub-opti}

In this subsection, we apply the  results from Subsection \ref{subsection-sub2} to develop optimality conditions for nonconvex-concave minimax problems by which we mean that in problem (\ref{minmaxg}), $f(x,y)$ is concave and $g(x,y)$ is convex in variable $y$ for all $x\in X$.

When $\bx$, a local optimal solution to the problem $(P_V)$,  is an interior point of the set $X$, it follows that $0\in \widehat{\partial} V(\bar x)$ by Fermat's rule. Thus, the Lipschitz continuity of the maximal value function is not necessary and the following  optimality condition holds under fairly weak assumptions.

\begin{thm}\label{thm-mi}
Suppose that $f(x,y)$ is concave and $g(x,y)$ is convex in variable $y$ for all $x\in X$ and $f(x,y), g(x,y)$ are second order continuously differentiable. Let $\bx \in {\rm int }X$ be a local minimum to the  problem $(P_V)$ and $\bar y \in {\cal S} (\bar x)$. Suppose that the KKT condition holds for problem $(P_\bx)$ at $\bar y$ with a multiplier $\bar \lambda$
and a CQ holds at  $(\bx,\by,\bar{\lambda})$ for the constraint system
\begin{equation}\label{constraintsystem-new}
\nabla_y \cL(x,y,\lambda) = 0,\ \lambda\ge0.
\end{equation}
Then there exists a vector $\bar u$ such that the following system holds:
\begin{eqnarray}\label{ld}
&& \begin{array}{l}
\nabla_x \cL(\bx,\by,\bar\lambda) + \nabla_{xy}^2 \cL(\bx,\by,\bar\lambda)\bar u = 0,\\ [5pt]
 \nabla_{yy}^2 \cL(\bx,\by,\bar\lambda)\bar u=0,\ \nabla_y \cL(\bx,\by,\bar\lambda)=0,\\ [5pt]
 0\le -\nabla_{y}g(\bx,\by)^\top \bar u- g(\bx,\by)\ \bot\ \bar{\lambda} \ge0,\
  0\le - g(\bx,\by)\ \bot\ \bar{\lambda} \ge0.\end{array}
 \end{eqnarray}
\end{thm}
\begin{proof}
By Fermat's rule and the optimality of $\bx$, it follows that $0\in \widehat{\partial} V(\bar x)$ since $\bar x$ is an interior point of $X$. Then the result follows from  Theorem \ref{thm sub} immediately.
\end{proof}

When $\bx$ is a boundary point of the set $X$, we do need the Lipschitz continuity of the maximal value function at $\bar x$. In this case the optimality condition holds under some extra assumptions which would guarantee the locally Lipschitz continuity of the value function.

\begin{thm}\label{thm-mima}
Suppose that $f(x,y)$ is concave and $g(x,y)$ is convex in variable $y$ for all $x\in X$ and $f(x,y), g(x,y)$ are second order continuously differentiable. Let $\bx\in {\rm bdy} X$ be a local solution to the problem $(P_V)$. Suppose that  for the inner problem, the restricted sup-compactness holds at $\bx$ and one of the following conditions hold:
\begin{itemize}
\item[(a)] MFCQ holds at all $y\in {\cal S}(\bx)$, {a stable CQ holds at all $(\bx, \tilde y,\tilde{\lambda})$ where $\tilde y \in {\cal S}(\bar x)$ and $\tilde{\lambda}\in \Sigma(\bx, \tilde y)$  for the constraint system (\ref{constraintsystem-new})};
\item[(b)] The RS holds at some $y\in \cS(\bx)$, a stable parametric CQ holds at all ${y}\in \cS(\bx)$
 and {a stable CQ} holds at all $(\bx, \tilde y,\tilde{\lambda})$ where $\tilde y \in {\cal S}(\bar x)$ and $\tilde{\lambda}\in \Sigma(\bx, \tilde y)$  for the constraint system (\ref{constraintsystem-new}).
Suppose that
\begin{equation}\label{singu-lip}
 \bigcup_{y\in {\cal S}(\bx), \lambda\in \Sigma(\bx,y)} \left\{
 \nabla_{xy}^2 \cL(\bx,y,{\lambda}) u:\begin{array}{l} \nabla_{yy}^2 \cL(\bx,y,{\lambda})u=0,\\[4pt]
0\le -\nabla_y g(\bx,y)^\top u \ \bot\ \lambda \ge 0
 \end{array}
 \right\} =\{0\}.
\end{equation}
\end{itemize}
Then there exists a point $\bar{y}\in {\cal S}(\bar x)$,  $\bar \lambda \in \Sigma(\bx,\by) $ and a multiplier $\bar u$ such that the following system holds.
\begin{eqnarray}\label{ldnew}
&& \begin{array}{l}
0\in \nabla_x \cL(\bx,\by,\bar\lambda) + \nabla_{xy}^2 \cL(\bx,\by,\bar\lambda)\bar u + \cN_X(\bx),\\ [5pt]
\nabla_y \cL(\bx,\by,\bar\lambda)=0, \  \nabla_{yy}^2 \cL(\bx,\by,\bar\lambda)\bar u=0,\\ [5pt]
 0\le -\nabla_{y}g(\bx,\by)^\top \bar u- g(\bx,\by)\ \bot\ \bar{\lambda} \ge0,\
  0\le - g(\bx,\by)\ \bot\ \bar{\lambda} \ge0.\end{array}
 \end{eqnarray}
\end{thm}
\begin{proof}
Since $\bar x\in \displaystyle \arg \min_{x\in X} V(x)$, by Fermat's rule,
$
0\in \partial (V+I_X)(\bx),
$ where $I_C(c)$ denotes the indicator function to set $C$ at $c$.

(a) By Clarke \cite[Corollary 1 to Theorem 6.5.2]{Clarke} under the restricted sup-compactness and  MFCQ condition,  $V(x)=-v(x)$ is Lipschitz continuous at $\bx$. Thus by   Rockafellar and Wets \cite[Corollary 10.10]{Rock98}, $0\in \partial V(\bx)+ \cN_X(\bx)$.  Noting that  MFCQ is a stable parametric CQ, the desired result follows immediately from Theorem \ref{thm sub}.

(b) By Corollary \ref{Cor4.2}, $V(x)$ is Lipschitz continuous at $\bx$ and then by Rockafellar and Wets \cite[Corollary 10.10]{Rock98}, $0\in \partial V(\bx)+ \cN_X(\bx)$. Then the desired result follows from Theorem \ref{thm sub} immediately.
\end{proof}

We give some comments on stationarity conditions derived by using MPEC approach (in Theorem \ref{thm-stampec}) and by the Wolfe duality approach (in Theorems \ref{thm-mi} and \ref{thm-mima}). First, we give an implication relationship as follows.
\begin{thm}\label{thm-com}
Let $\bx\in X$. Assume that there exist $\by\in \cS(\bx)$, $\lambda\in \Sigma(\bx,\by)$ and a multiplier $\bar u$ such that \eqref{ldnew} holds. Then $(\bx,\by,\bar\lambda)$ is an S-stationary point of ${\rm (P_{MPEC})}$.
\end{thm}
\begin{proof}
Assume that condition \eqref{ldnew} holds. Let $\bar\alpha = -\bar\lambda$ and $\bar\beta = -\nabla_{y}g(\bx,\by)^\top \bar u$. By these definitions and \eqref{ldnew},  the third and fourth lines of \eqref{weak-sta} hold. We also note that
\begin{eqnarray*}
&& \nabla_xf(\bx,\by) + \nabla_x g(\bx,\by)\bar\alpha = \nabla_xf(\bx,\by) - \nabla_x g(\bx,\by)\bar\lambda = \nabla_x \cL(\bx,\by,\bar\lambda),\\
&& \nabla_yf(\bx,\by) + \nabla_y g(\bx,\by)\bar\alpha = \nabla_yf(\bx,\by) - \nabla_y g(\bx,\by)\bar\lambda = \nabla_y \cL(\bx,\by,\bar\lambda)=0.
\end{eqnarray*}
Thus, it follows from these two conditions and condition \eqref{ldnew} that the first and second lines of \eqref{weak-sta} hold. By the complementarity condition $0\le - g(\bx,\by)\ \bot\ \bar{\lambda} \ge0$, it follows that $\bar\alpha_i = -\lambda_i = 0$ if $g_i(\bx,\by)<0$. By the two complementarity conditions in last line of condition \eqref{ldnew}, it follows that $-\nabla_{y}g_i(\bx,\by)^\top \bar u- g_i(\bx,\by)=0$ and  $g_i(\bx,\by)=0$ if $\bar\lambda_i<0$, $-\nabla_{y}g_i(\bx,\by)^\top \bar u\ge0$ if $g_i(\bx,\by)=\bar\lambda_i=0$. These imply that $\bar\beta_i = -\nabla_{y}g_i(\bx,\by)^\top \bar u =0$ if $\bar\lambda_i<0$ and $\bar\beta_i = -\nabla_{y}g_i(\bx,\by)^\top \bar u \ge 0$ if $g_i(\bx,\by)=\bar\lambda_i=0$. Thus, condition \eqref{s-multiplier} and the last line of condition \eqref{weak-sta} hold.
\end{proof}

Theorem \ref{thm-com} shows that the stationarity condition derived by using by the Wolfe duality approach is stronger than that derived by using the MPEC approach. We now compare the required sufficient conditions. MPEC-LICQ required in Theorem \ref{thm-stampec} is a very stringent condition as discussed after Theorem \ref{thm-stampec} while the conditions required in Theorems \ref{thm-mi} and \ref{thm-mima} are commonly used for nonlinear programming problems.

Based on Corollaries  \ref{cor-regu} and \ref{cor2}, we have the following result.
\begin{thm}Suppose that $f(x,y)$ is concave and $g(x,y)$ is convex in variable $y$ for all $x\in X$ and $f(x,y), g(x,y)$ are second order continuously differentiable.
Let $\bx$ be a local solution to the problem $(P_V)$.
\begin{itemize}
\item[ (i)] Assume that $\bar x\in {\rm int} X$,  $\bar y \in {\cal S} (\bar x)$ and  $\bar{\lambda}\in \Sigma(\bx,\by)$. Suppose that either $f(x,y)$ and $g(x,y)$  are separable with respect to $x$ and $y$ and a CQ holds at  $(\bx,\by,\bar{\lambda})$ for the constraint system \eqref{constraintsystem-new}
or SOSC holds at $(\bar y,\bar\lambda)$ for problem $(P_\bx)$. Then there exists $\bar \lambda$ such that the following system holds:
$$0= \nabla_x \cL(\bx,\by,\bar\lambda),\
 \nabla_y \cL(\bx,\by,\bar\lambda)=0,\ 0\le - g(\bx,\by)\ \bot\ \bar{\lambda} \ge0.$$

\item[(ii)] Assume that $\bar x\in {\rm bdy} X$. Suppose that  for the inner problem, the restricted sup-compactness holds at $\bx$,   RS holds at some $y\in \cS(\bx)$ and  a stable parametric CQ holds at all $y\in \cS(\bx)$. If either $f(x,y)$ and $g(x,y)$ are separable with respect to $x$ and $y$ and a stable CQ holds at all $(\bx,\tilde y,\tilde{\lambda})$ for the constraint system (\ref{constraintsystem-new}) where $\tilde y\in {\cal S}(\bar x)$ and  $\tilde{\lambda}\in \Sigma(\bx,\tilde y)$, or SOSC holds at all $(\tilde y,\tilde{\lambda})$ where $\tilde y\in {\cal S}(\bar x)$ and  $\tilde{\lambda}\in \Sigma(\bx,\tilde y)$, then there exist $\bar y\in \cS(\bx)$ and $\bar \lambda$ such that
\begin{equation}\label{op}
 0\in \nabla_x \cL(\bx,\by,\bar\lambda)+\cN_X(\bar x),\
 \nabla_y \cL(\bx,\by,\bar\lambda)=0,\ 0\le - g(\bx,\by)\ \bot\ \bar{\lambda} \ge0.
 \end{equation}
 \end{itemize}
\end{thm}
\begin{proof}
The result (i) follows from Corollary  \ref{cor-regu} and the fact that $0\in \widehat{\partial}V(\bx)$ immediately. Assume that $\bar x\in {\rm bdy} X$. By the optimality of $\bx$, we have $0\in \partial (V + I_X)(\bx)$. Then by Corollary \ref{cor2}, it follows that $V$ is Lipschitz continuous at $\bx$. Thus, $0\in \partial V(\bx) + \cN_X(\bx)$. The results in (ii) follow from  Corollary \ref{cor2}.
\end{proof}

\subsection{Optimality conditions for nonconvex-nonconcave minimax problems}

In this section we discuss the general nonconvex-nonconcave case.
Let $(\bar x,\bar y)$ be a local solution to the  general nonconvex-nonconcave minimax problem in the sense of Stackelberg. If the value function $V(x)$ is Lipschitz continuous, then
\begin{equation}\label{convexhull}
0\in   \co \bigcup_{y\in {\cal S}(\bar x)} \left \{\nabla_x \cL(\bx,y,\bar{\lambda}): \bar{\lambda} \in \Sigma(\bar x, y) \right \} +\cN_X(\bx).\end{equation}
Moreover by Caratheodory's theorem, the convex hull on the set ${\cal S}(\bar x)$ can be represented by the convex hull of no more than $n+1$ points in ${\cal S}(\bar x)$.
This is the type of necessary optimality condition proposed in Ye and Zhu \cite[Theorem 4.1]{Ye-Zhu}.
Note that it is easy to see that the restricted sup-compactness condition was missed in the statement of \cite[Theorem 4.1]{Ye-Zhu}.

In this subsection we will derive optimality conditions under which (\ref{convexhull}) is replaced by
$$ 0\in  \left \{\nabla_x \cL(\bx,\by,\bar{\lambda}): \bar{\lambda} \in \Sigma(\bar x, \by)  \right \}+\cN_X(\bx),$$
where $\by$ is a given solution chosen from  ${\cal S}(\bx)$.

The first result is based on the inner semi-continuity of the solution map $\cS(x)$.
\begin{thm}\label{thm5.3}
Let $(\bx,\by)$ be a local solution to the minimax problem $(P_{\rm minimax})$. Suppose that  the solution map
${\cal S}(x)$ is inner semi-continuous at $(\bar x,\bar y)$,  the local error bound holds for the system $g(x,y)\leq 0$ at $(\bar x, \bar y)$, and
 \begin{equation}\label{sigular2}
 \{\nabla_x g(\bx,\by) {\lambda}:{\lambda} \in \Sigma^0(\bar x,\bar y)\}=\{0\}.
 \end{equation}
Then there exists $\bar \lambda$ such that the system \eqref{op} holds.
\end{thm}
\begin{proof} 
Under the assumptions, by Theorem \ref{thm4.1} it follows that $V(x)$ is Lipschitz continuous at $\bx$ and
\[
\partial V(\bar x) \subseteq \{\nabla_x \cL(\bx,\by,\bar{\lambda}) : \bar{\lambda} \in \Sigma(\bar x,\bar y)\}
\]
The proof is complete by combining the above inclusion with Fermat's rule (\ref{Fermat}).
\end{proof}

The following corollary follows from Corollary \ref{cor4.1} and  Fermat's rule (\ref{Fermat}) immediately.

\begin{cor}Let $(\bx,\by)$ be a local solution to the minimax problem $(P_{\rm minimax})$.
Suppose that  the solution map
${\cal S}(x)$ is inner semi-continuous at $(\bar x,\bar y)$,  and  MFCQ  for the system $g(\bx,y)\leq 0$ holds at $y=\by$.
 Then there exists $\bar \lambda$  such that the system (\ref{op}) holds.
\end{cor}

The second result is based on the concavity of the value function $V(x)$.
\begin{thm}\label{thm5.4}
Let $(\bx,\by)$ be a local solution to the minimax problem $(P_{\rm minimax})$.
Suppose that $f(x,y)$ is jointly concave and $g(x,y)$ is jointly quasiconvex in $(x,y)$, and there is an open set ${\cal O} \ni \bar x$ such that $\cF(x)$ is nonempty and the objective function $f(x,y)$ is bounded above on $\cF(x)$. Assume further that a CQ holds for $g(x,y)\leq 0$ at $(\bx,\by)$.
Then there exists $\bar \lambda$ such that the system (\ref{op}) holds.
\end{thm}
\begin{proof} 
By Theorem \ref{thm4.2}, $V(x)$ is concave and  Lipschitz continuous at $\bx$.
Moreover
\[
\partial V(\bar x) \subseteq \{\nabla_x \cL(\bx,\by,\bar{\lambda}) : \bar{\lambda} \in \Sigma(\bar x,\bar y)\}.
\]
The proof is therefore  complete by combing the above inclusion with  Fermat's rule (\ref{Fermat}).
\end{proof}

\subsection{Applications in generative adversarial networks}

In this section, we will apply our new optimality conditions to the  adversarial modeling framework called  generative adversarial networks (GAN) first proposed by Goodfellow et al.  \cite{GAN2014}.
While the generative network  generates new samples, the discriminative network  tries to distinguish them from the true data.
Let $G(z;x)$ be a generator parameterized by $x$ and $D(s;y)$ be the probability that $s$ came from the data rather than the generator's distribution over data, where $y$ is the discriminator's parameter. Given $x$,  the discriminative network  is trained to maximize the probability of assigning the correct label to both training examples and samples from the generative network.  Thus, the minimax problem corresponding to GAN can be written as
\begin{equation}\label{gan}
	\min_{x}\max_{y}\, f(x,y):=\mathbb{E}_{s\sim p_{data}(s)}\left[\log D(s;y)\right]+\mathbb{E}_{z\sim p_{latent}(z)}\left[\log\left(1- D(G(z;x);y)\right)\right],
\end{equation}
where $p_{data}$ is the data distribution, and $p_{latent}$ is the latent distribution.

Consider the situation in which the generator  is a single-layer network with the identity activation function and the discriminator  is also a single-layer network with a logistic sigmoid activation function. That is,
\begin{equation}
	G(z;x)=x+z,\quad D(s;y)=\frac{1}{1+e^{y^T(s-\bar{s})}},
\end{equation}
where $x\in\mathbb{R}^n$ is the parameter of the generator, $y\in\mathbb{R}^n$ is the parameter of the discriminator, and $\bar{s}=\mathbb{E}_{s\sim p_{data}}[s]$ is the average vector over the dataset, cf. the simple CIFAR GANs in Grimmer et al. \cite{GLWM2022}.

For simplicity, generating one sample $s$ from $p_{data}$ and $z$ from $p_{latent}$ repectively, we have
\begin{align*}
	f(x,y)=&\log\left(\frac{1}{1+e^{y^T(s-\bar{s})}}\right)
	+\log\left(1-\frac{1}{1+e^{y^T(x+z-\bar{s})}}\right)\\
	=&y^T(x+z-\bar{s})-\log\left(1+e^{y^T(s-\bar{s})}\right)-\log\left(1+e^{y^T(x+z-\bar{s})}\right).
\end{align*}
 Note that
$$\nabla_x f(x,y)=\frac{1}{1+e^{y^T(x+z-\bar{s})}} y,
 \nabla_y f(x,y)=\frac{1}{1+e^{y^T(x+z-\bar{s})}}(x+z-\bar{s})-\frac{e^{y^T(s-\bar{s})}}{1+e^{y^T(s-\bar{s})}}(s-\bar{s}).
$$
\begin{equation*}
	\nabla_{xx}^2 f(x,y)=-\frac{e^{y^T(x+z-\bar{s})}}{\left[1+e^{y^T(x+z-\bar{s})}\right]^2} y y^T \preceq 0,
\end{equation*}
\begin{equation*}
	\nabla_{yy}^2 f(x,y)=-\frac{e^{y^T(x+z-\bar{s})}}{\left[1+e^{y^T(x+z-\bar{s})}\right]^2}(x+z-\bar{s})(x+z-\bar{s})^T
	-\frac{e^{y^T(s-\bar{s})}}{\left[1+e^{y^T(s-\bar{s})}\right]^2}(s-\bar{s})(s-\bar{s})^T\preceq 0.
\end{equation*}
Thus $\min_{x}\max_{y}f(x,y)$ is a nonconvex-concave minimax problem. Define $\bar{x}=s-z$. we have
$
\nabla_y f(\bar{x},y)=\frac{1-e^{y^T(s-\bar{s})}}{1+e^{y^T(s-\bar{s})}}(s-\bar{s})
$ and hence $
\nabla_y f(\bar{x},y)=0$ if and only if $y^T(s-\bar{s})=0$. By the
concavity of $f(x,y)$  in $y$ for every $x$, we  get   $\mathcal{S}(\bar{x})=\left\{y\in\mathbb{R}^n:y^T(s-\bar{s})=0\right\}$ and then $\max_{y}f(\bar{x},y)=-2\log 2$. On the other hand, note that $f(x,0)=-2\log 2$ and $\nabla_y f(x,0)=\frac{1}{2}(x+z-s)$. Hence $\nabla_y f(x,0)\neq 0 $ for all $x\neq\bar{x}$ and so
\begin{equation}\label{minimaxy0}
	\max_{y}f(x,y)>f(x,0)=-2\log 2=\max_{y}f(\bar{x},y) \mbox{  for all }x\neq\bar{x}.
\end{equation}
 Thus all of the optimal solutions of $\min_{x}\max_{y}f(x,y)$ are given by $\{(\bar{x},\bar{y}):\bar{y}\in\mathcal{S}(\bar{x})\}$.

We claim that the optimality condition in Theorem 4.3 holds for all of the optimal solutions of the above example, but the commonly used first-order condition in the sense of Nash equilibrium
\begin{equation} \label{stationary}
\nabla_x f(\bar{x},\bar{y})=0,\quad \nabla_y f(\bar{x},\bar{y})=0\end{equation}
 only applies to one of them, that is the optimal solution $(\bar{x},0)$. Recall that the stationary condition (\ref{stationary})  is widely used in characterizing (local) Nash equilirium and local minimax point defined in Jin et al. \cite{Jin} and used in the convergence analysis of various kinds of gradient descent ascent algorithms for minimax problems, cf. e.g., Definition 2.1 and Theorem 3.4 in Ozdaglar and Bertsekas \cite{OzBer}, Proposition 3.1 and Theorems 1 and 2 in Lu et al.  \cite{LTH2020}, Definition 4.10 and Propositions 4.11 and 4.12 in  Lin et al. \cite{LJJ2020}, Definition 5 and Theorem 12 in  Lin et al. \cite{LJJ2020b}, Definition 1 and Theorem 3.2 in Yang et al.\cite{YKH2020}, Definition 2.1, Remark 2.2 and Theorem 4.1 in Ostrovskii et al.\cite{OLR2021}. Next we prove our claim. Indeed, on the one hand,
 $\nabla_x f(\bar{x},\bar{y})=0$ implies that $\bar{y}=0$. This means that the optimal solutions $(\bar{x},\bar{y})$ with $0\neq\bar{y}\in\mathcal{S}(\bar{x})$ cannot be characterized by the usual first-order condition (\ref{stationary}).
On the other hand, it is easy to verify that the necessary conditions in Theorem \ref{thm-mi} would be satisfied, i.e., for any $\bar{y}\in\mathcal{S}(\bar{x})$, there exists $\bar{u}=-\bar{y}$ such that
\begin{align}
	&\nabla_x f(\bar{x},\bar{y})+\nabla_{xy}^2 f(\bar{x},\bar{y}) \bar{u}=0,\quad \nabla_{yy}^2 f(\bar{x},\bar{y}) \bar{u}=0,\quad \nabla_y f(\bar{x},\bar{y})=0\label{optimalityc}.
\end{align}

Furthermore, some constraint functions could be added in GAN~\eqref{gan} such that MPEC-LICQ fails while the conditions required in Theorem~\ref{thm-mi} hold. For example, taking $g(x,y)=\big(\frac{1}{2}\|y\|^2-\frac{1}{2}, y_1-1\big)^T$, $s-\bar{s}=(0,1)^T$, and $\bar{y}=(1,0)^T\in\mathbb{R}^2$, then $\bar x=s-z, \bar y=(1,0)^T$ is still an optimal solution because both $y=\bar y$ and $y=0$ satisfy the constraints $g(x,y)\leq0$ for all $x$ and the inequality \eqref{minimaxy0} holds.  Since
\begin{equation*}
	\nabla_{y} f(x,y)-\nabla_{y} g(x,y)^T \lambda
	=\frac{1}{1+e^{y^T(x+z-\bar{s})}}(x+z-\bar{s})-\frac{e^{y^T(s-\bar{s})}}{1+e^{y^T(s-\bar{s})}}(s-\bar{s})
	-\lambda_1 y- \lambda_2 (1,0)^T,
\end{equation*}
we have
$$\nabla_{y} f(\bar x,y)-\nabla_{y} g(\bar x,y)^T \lambda=-\lambda_1 y- \lambda_2 (1,0)^T.$$
Hence one can easily obtain $\Sigma(\bar{x},\bar{y})=\{(0,0)\}$. Since $\cI_{00}=\{1,2\}$, the vectors $\nabla_{(x,y,\lambda)} g_1(\bar{x},\bar{y})=(0,0,1,0,0,0)^T$ and $\nabla_{(x,y,\lambda)} g_2(\bar{x},\bar{y})=(0,0,1,0,0,0)^T$ are linearly dependent, MPEC-LICQ defined in Theorem \ref{thm-stampec} fails. But LICQ holds at $(\bar{x},\bar{y},0)$ for the constraint system \eqref{constraintsystem-new} since the matrix
\begin{equation*}
	\left(\nabla_{x,y,\lambda}\nabla_y\mathcal{L}(x,y,\lambda), \nabla_{x,y,\lambda} \lambda\right)
	=
\begin{bmatrix}
	\frac{1}{2} & -\frac{1}{4} & 0 & 0  \\
	0 & \frac{1}{2} & 0 & 0\\
	0 & 0 & 0 & 0 \\
	0 & -\frac{1}{2} & 0 & 0\\
	-1 & 0 & 1 & 0\\
	-1 & 0 &0 & 1
\end{bmatrix}
\end{equation*} has full column rank.
Since the multiplier $\bar \lambda_1=0, \bar \lambda_2=0$, the necessary optimality condition for the constrained case is the same as the unconstrained one in (\ref{optimalityc}).
This illustrates that  the stationarity condition  in Theorem~\ref{thm-com} by using the Wolfe duality approach is sharper than the one derived by using the MPEC approach while the conditions required in Theorems~\ref{thm-mi} and \ref{thm-mima} are, in general, weaker than MPEC-LICQ.

\section{Concluding remarks}\label{sec:con}

In this paper, we derived upper estimates for the Fr\'{e}chet, limiting, and horizon subdifferentials of the maximal value function with the union of all solutions or only one solution from the solution set.
Based on the derived results, we developed necessary optimality conditions for nonconvex (non-)concave minimax problems.


\section*{Acknowledgments}

The authors are grateful to the associated editor and two referees for their helpful comments and constructive suggestions.  In particular, we thank one of the referees for suggesting the sufficient condition for the continuity of the value function that is weaker than RS which we used in an earlier version of the paper.

\end{document}